\documentclass[12pt]{amsart}
\usepackage{latexsym,amssymb,amsthm,amsmath,enumerate,graphicx,psfrag,mathrsfs,longtable,cite,stmaryrd,fullpage}
\usepackage{bbm}
\usepackage{rotating}
\usepackage{rotate}
\usepackage{pst-node}
\usepackage{pstricks}
\usepackage{srcltx}
\usepackage{psfrag}
\usepackage[all,cmtip]{xy}
\begin{document}

\theoremstyle{plain}
  \newtheorem{theorem}{Theorem}[section]
  \newtheorem{proposition}[theorem]{Proposition}
  \newtheorem{lemma}[theorem]{Lemma}
  \newtheorem{corollary}[theorem]{Corollary}
  \newtheorem{conjecture}[theorem]{Conjecture}
\theoremstyle{definition}
  \newtheorem{definition}[theorem]{Definition}
  \newtheorem{example}[theorem]{Example}
  \newtheorem{observation}[theorem]{Observation}
  \newtheorem{convention}[theorem]{Convention}
  \newtheorem{observations}[theorem]{Observations}
  \newtheorem{question}[theorem]{Question}
 \theoremstyle{remark}
  \newtheorem{remark}[theorem]{Remark}
\def\ff{{\mathbf f}}
\def\ZZ{{\mathbb Z}}
\def\QQ{{\mathbb Q}}
\def\CC{{\mathbb C}}
\def\NN{{\mathbb N}}
\def\comp{\mathrm{c}}
\def\Kbar{{\overline{K}}}
\def\isom{\cong}
\def\homotopic{\simeq}

\def\Y{{\mathbf Y}}
\def\SY{{S\mathbf Y}}
\def\YF{{\mathbf Y F}}
\def\il{{\text{Smith-eigenvalue}}}
\def\pd{{\frac{\partial}{\partial p_1}}}
\def\K{{\mathbb C}}
\def\E{{\mathcal E}}
\def\A{{\mathscr A}}
\def\P{{\mathscr{P}}}

\def\Q{{\mathbf Q}}
\def\GL{\mathrm{GL}}
\def\symm{\mathfrak{S}}

\def\tilde{\widetilde}

\def\diff{\setminus}
\def\Par{P(W,R)}
\def\coker{\mathrm{coker}}
\def\im{\mathrm{im}}
\def\diag{\mathrm{diag}}
\def\rank{\mathrm{rank}}
\def\Res{\mathrm{Res}}
\def\Ind{\mathrm{Ind}}
\def\triv{{\mathbbm{1}}}
\def\Irr{\mathrm{Irr}}
\def\ones{{\mathrm{ones}}}
\def\Gal{{\mathrm{Gal}}}
\def\Fix{\mathrm{Fix}}
\def\Pw{\mathcal P_W}
\def\type{\mathrm{type}}
\def\Class{{\mathbf{Cl}}}
\def\CoShe{\mathrm{CoShe}}
\def\sd{\mathrm{sd}}
\def\pDelta{\Delta^{U}}
\def\lk{\mathrm{lk}}
\def\L{\mathscr L}
\def\calF{\mathcal{F}}
\def\calP{\mathcal{U}}
\def\calP{\mathcal{U}}
\def\calB{\mathcal{B}}
\def\calT{T}
\def\calJ{\mathcal{J}}
\def\affine{\mathrm{{affine}}}
\def\pPi{\Pi^U}
\def\Stab{\mathrm{Stab}}
\def\St{\mathrm{St}}
\def\Supp{\mathrm{Supp}}
\def\Span{\mathrm{Span}}
\def\Def{\stackrel{\textbf{def}}{=}}
\def\GDef{\stackrel{\phantom{\text{def}}}{=}}
\def\Dash{\text{\textbf{---}}}
\def\Conv{\mathrm{Conv}_{\mathbb R}}
\def\AffSpan{\mathrm{AffSpan}}
\def\LinSpan{\mathrm{Span}}
\def\Hull{\mathrm{Hull}}
\def\Face{\mathrm{Face}}
\def\Sd{\mathrm{Sd}}
\def\Hilb{\mathrm{Hilb}}
\def\Hom{\mathrm{Hom}}
\newcommand{\Wedge}{{\textstyle{\bigwedge}}}
\title{Reflection arrangements and ribbon representations}
\author{Alexander R. Miller}
\email{mill1966@math.umn.edu}
\address{ School of Mathematics\\
University of Minnesota\\
Minneapolis, MN 55455} 
\begin{abstract}
Ehrenborg and Jung~\cite{Ehrenborg} recently related the order complex for the lattice of 
$d$-divisible partitions with the simplicial complex of \emph{pointed ordered set partitions} 
via a homotopy equivalence.  The latter has top homology naturally identified as a 
Specht module.  Their work unifies that of Calderbank, Hanlon, Robinson~\cite{CalHanRob}, and Wachs~\cite{Wachs}.
By focusing on the underlying geometry, we strengthen and extend these results 
from type $A$ to all real reflection groups and the complex reflection groups known as 
\emph{Shephard groups}.
\end{abstract}
\thanks{Supported by NSF grant DMS-1001933}
\maketitle
\section{Introduction}
The aim of this paper is to elucidate a phenomenon that has been studied 
for the symmetric group $\mathfrak S_n$ by studying the underlying geometry.  Here we sketch 
the phenomenon, along with our geometric interpretation and generalization.

For $n+1$ divisible by $d$, recall that the $d$-divisible partition lattice $\Pi_{n+1}^d\cup\{\hat{0}\}$ 
is the poset of partitions of the set $\{1,2,\ldots, n+1\}$ with 
parts divisible by $d$, together with a unique minimal element $\hat{0}$ when $d>1$.  
In~\cite{CalHanRob}, Calderbank, Hanlon and Robinson showed that for $d>1$ the top homology of 
the order complex 
$\Delta(\Pi_{n+1}^d\diff\{\hat{1}\})$, when restricted from $\mathfrak S_{n+1}$ to 
$\mathfrak S_n$, carries the \emph{ribbon representation} of $\mathfrak S_n$ 
corresponding to a ribbon with row sizes $(d,d,\ldots, d, d-1)$.  
Wachs~\cite{Wachs} gave a more explicit proof of this fact.  
Their results generalized Stanley's~\cite{Stanley} result for the M\"obius 
function of $\Pi^d_n\cup\{\hat{0}\}$, which generalized G. S. Sylvester's~\cite{Sylvester} result for 
2-divisible partitions $\Pi^2_n\cup\{\hat{0}\}$.

Ehrenborg and Jung extend the above 
results by introducing posets of \emph{pointed partitions} $\Pi_{\vec{c}}^\bullet$ 
parametrized by a composition $\vec{c}$ of $n$ with last part possibly $0$, 
from which they obtain all ribbon representations.  More importantly, they explain why 
Specht modules are appearing by establishing a homotopy equivalence with another complex whose top homology is 
naturally a Specht module.

Ehrenborg and Jung construct their pointed partitions $\Pi_{\vec{c}}^\bullet\subset\Pi_{n}^\bullet\cong \Pi_{n+1}$ 
by distinguishing a particular block (called the \emph{pointed block}) and 
restricting to those of type $\vec{c}$.  
They show that $\Delta(\Pi_{\vec{c}}^\bullet\diff\{\hat{1}\})$ is homotopy equivalent 
to a wedge of spheres, and that the top reduced homology 
$\tilde{H}_{\rm{top}}(\Delta(\Pi_{\vec{c}}^\bullet\diff\{\hat{1}\}))$ is the $\mathfrak S_n$-Specht 
module corresponding to $\vec{c}$.  

Their approach is to first relate $\Pi_{\vec{c}}^\bullet$ to 
a selected subcomplex $\Delta_{\vec{c}}$ of the simplicial complex $\Delta_n^\bullet$ 
of ordered set partitions of $\{1,2,\ldots, n\}$ with last block possibly empty.  
In particular, they use Quillen's fiber lemma 
to show that $\Delta(\Pi_{\vec{c}}^\bullet\diff\{\hat{1}\})$ is homotopy equivalent to 
$\Delta_{\vec{c}}$.  They then 
give an explicit basis for $\tilde{H}_{\rm{top}}(\Delta_{\vec{c}})$ that identifies the top homology as a Specht module.

Ehrenborg and Jung recover the results of Calderbank, Hanlon and 
Robinson~\cite{CalHanRob} and Wachs~\cite{Wachs} by specializing to 
$\vec{c} =(d,\ldots,d,d-1)$.

Taking a geometric viewpoint, one can consider $\Delta_n^\bullet$ as the barycentric subdivision 
of a distinguished facet of the standard $n$-simplex having 
vertices labeled with $\{1,2,\ldots, n,n+1\}$.  
As such, it carries an action of $\mathfrak S_n$ and is a balanced 
simplicial complex,  with each $\Delta_{\vec{c}}$ corresponding to a 
particular type-selected subcomplex.
Under this identification, the poset $\Pi^\bullet_{\vec{c}}$ 
corresponds to linear subspaces spanned by faces in $\Delta_{\vec{c}}$.  

We propose an analogous program for all well-generated complex reflection groups by introducing \emph{well-framed} and \emph{locally conical} 
systems.  We complete the program for all irreducible \emph{finite} groups having 
a presentation of the form
\begin{equation}\label{CoxeterShephard}
\langle r_1,\ldots, r_\ell\ \mid \ r_i^{p_i}=1,\quad \underbrace{r_ir_{j}r_i\cdots}_{q_{ij}}=\underbrace{r_{j}r_ir_{j}\cdots}_{q_{ij}}\quad\forall i,j\rangle
\end{equation}
with $p_i\geq 2$ for all $i$.  Each such group has an irreducible faithful representation as a complex reflection group.  The irreducible 
finite Coxeter groups are precisely those with each $p_i=2$, i.e., those with a real form.  
The remaining groups are  \emph{Shephard groups}, the 
symmetry groups of regular complex polytopes.\footnote{The algebraic unification of Coxeter groups and Shephard groups presented here does not appear to be widely known, and is 
attributed to Koster~\cite[p. 206]{OrlikSolomon}.}  The family of Coxeter and Shephard 
groups contains 21 of the 26 exceptional well-generated complex reflection groups.  
Using Shephard and Todd's numbering, the remaining five groups are $G_{24},G_{27},G_{29},G_{33},G_{34}$.

\tableofcontents

\section{Well-framed systems for complex reflection groups}\label{Section:Well-framed}
Let $V$ denote an $\ell$-dimensional $\mathbb C$-vector space.  A \emph{reflection} in 
$V$ is any non-identity element $g\in \GL(V)$ of finite order 
that fixes some hyperplane $H$, and   
a finite group $W\subset \GL(V)$ is called a \emph{reflection group} if 
it is generated by reflections.  Henceforth, we assume that $W$ acts irreducibly on $V$.  
Shephard and Todd gave a complete classification 
of all such groups in~\cite{ShephardTodd}.
  
Given a (finite) reflection group $W\subset \GL(V)$, we may choose a positive definite 
Hermitian form $\langle-,-\rangle$ on $V$ 
that is preserved by $W$, i.e.,  
\[\langle gx,gy\rangle=\langle x,y\rangle\]
for all $x,y\in V$ and $g\in W$.  We always regard $V$ as being 
endowed with such a form, which is unique up 
to positive real scalar when $W$ acts irreducibly on $V$.  We let $|-|$ 
denote the associated norm, defined by $|v|^2=\langle v,v\rangle$ for all $v\in V$.

As a special case of a complex reflection group, 
consider a finite group $W\subset \GL(\mathbb R^\ell)$ that is generated by 
reflections through hyperplanes.  By extending scalars, we consider $W$ as acting 
on $\mathbb C^\ell$, and regard $W$ as a reflection group.  
We will call a (complex) reflection group that arises in this way a 
\emph{(finite) real reflection group}.

A subgroup $W\subset \GL(V)$ naturally acts on the dual space $V^*$ via 
$gf(v)=f(g^{-1}v)$, and this action extends to the symmetric algebra $S=S(V^*)$.  
An important subalgebra of $S$ is \emph{the ring of invariants} $S^W$, whose 
structure actually characterizes reflection groups:

\begin{theorem}[Shephard-Todd, Chevalley]
A finite group $W\subset \GL(V)$ is a reflection group if and only if 
$S^W$ is a polynomial algebra $S^W=\mathbb C[f_1,\ldots,f_\ell]$.
\end{theorem}

When $W$ is a reflection group, such generators $f_1,\ldots, f_\ell$ of 
algebraically independent homogeneous polynomials for $S^W$ are not unique, 
but we do have uniqueness for the corresponding \emph{degrees} $d_1\leq\cdots\leq d_\ell$.

It is well-known that the minimum number of reflections required to generate $W$ is either 
$\dim V$ or $\dim V+1$.  If there exists such a generating set $R$ with $|R|=\dim V$, we say that $W$ is \emph{well-generated} and 
that $(W,R)$ is a \emph{well-generated system}.  (Finite) real reflection groups form an important class of 
well-generated reflection groups.  Another important family consists of symmetry groups 
of \emph{regular complex polytopes}; these are known as \emph{Shephard groups}, and 
were extensively studied by both Shephard~\cite{Shephard} and Coxeter~\cite{Coxeter}.

The (complex) reflecting hyperplanes of a real reflection group $W$ 
intersect the embedded real sphere
\[
\mathbb S^{\ell-1}=\{x\in \mathbb R^\ell\subset \mathbb C^\ell\ :\ |x|=1\}
\]
to form the \emph{Coxeter complex}, a simplicial complex with many wonderful properties.  
We call the maximal simplices of a Coxeter complex \emph{chambers}.  The real 
cone $\mathbb R_{\geq 0}C$ over a chamber $C$ is called a \emph{(closed) Weyl chamber}, and 
one nice feature of the Coxeter complex is its algebraic description when $(W,R)$ is 
a \emph{simple system}, i.e., when $R$ is the set of reflections through walls of a 
Weyl chamber.  
In this case, 
the poset of faces 
for the complex has the alternate description~\cite{Brown} as the poset of parabolic cosets 
\[\Delta(W,R)=\{gW_J\ :\ g\in W, J\subseteq R\},\]
ordered by reverse inclusion, i.e., 
\[gW_J<g'W_{J'}\quad \text{if}\quad gW_J\supset g'W_{J'}.\]

Naturally, one can define such a poset $\Delta(W,R)$ for an arbitrary group 
$W$ with set of \emph{distinguished generators} $R$.  In~\cite{BabsonReiner}, 
Babson and Reiner show that 
the geometry of this general construction is still well-behaved when $R$ is finite and minimal 
with respect to inclusion.  In particular, they show that such posets $\Delta(W,R)$ are 
\emph{simplicial posets}, meaning that every lower interval is isomorphic to a Boolean algebra.  
In fact, each 
$\Delta(W,R)$ is \emph{pure} of dimension $|R|-1$ and \emph{balanced}.  In other words, 
there is a coloring of the atoms using 
$|R|$ colors so that 
each maximal element lies above exactly one atom of each color.  The natural coloring given by
\begin{align*}
gW_{R\diff \{r_i\}}&\longmapsto \{r_i\}
\intertext{extends to a \emph{type function}}
\type:\Delta(W,R)&\longrightarrow \{\text{subsets of $R$}\}\\
gW_{R\diff J}&\longmapsto J.
\end{align*}
Recall that when such a coloring is present, we can select subposets by 
restricting to particular colors.  
Precisely, from each subset $T\subseteq R$, we obtain the following subposet 
\emph{selected by $T$}:
\begin{align*}
\Delta_T(W,R) &\Def \{\sigma\in \Delta(W,R)\ :\ \type(\sigma)\subseteq T\}\\
&\stackrel{\phantom{\text{Def}}}{=}\{g W_{R\diff J}\ :\ J\subseteq T\}.
\end{align*}
We will often write $\Delta$ in place of $\Delta(W,R)$, and $\Delta_T$ in place of 
$\Delta_T(W,R)$.

When $W$ is a reflection group, the lattice of intersections of reflecting hyperplanes 
for $W$ under reverse inclusion is denoted $\L_W$, or simply $\L$.  It is a 
subposet of the lattice $\L(V)$ 
of all $\mathbb C$-linear subspaces of $V$ ordered by reverse inclusion.  For 
a subset $A\subseteq V$, we define
\begin{align*}
\LinSpan(A)&\Def \text{minimal $\mathbb C$-linear subspace of $V$ containing $A$}.\\
\AffSpan(A)&\Def \text{minimal $\mathbb C$-linear affine space of $V$ containing $A$}.\\
\Hull(A)&\Def \left\{\sum_{i=1}^m t_i a_i\mid a_i\in A,\ t_i\geq 0,\ \sum_{i=1}^m t_i=1\right\}.
\end{align*}
These are the only notions of span and hull that will appear in this paper.

The following definitions aim to extend the relation between 
$\L$ and $\Delta(W,R)$ for simple systems of real reflection groups to well-generated 
systems for complex reflection groups.

\begin{definition}  
For $(W,R)$ a well-generated system, a \emph{frame} 
\[\Lambda=\{\lambda_1,\lambda_2,\ldots,\lambda_\ell\}\] 
is a collection of nonzero vectors with
\[\lambda_i\in H_1\cap \cdots\cap\hat{H_i}\cap\cdots\cap H_\ell\quad\text{for}\quad 
1\leq i\leq \ell.\]
Here, $H_i$ is the reflecting hyperplane for reflection $r_i\in R$.  We say that 
$(W,R,\Lambda)$ is a \emph{framed system}.
\end{definition}

We will sometimes index a generating set $R$ and frame $\Lambda$ 
with $\{0,1,\ldots, \ell-1\}$ instead of $\{1,2\ldots, \ell\}$, 
writing $R=\{r_0,r_1,\ldots, r_{\ell-1}\}$ and 
$\Lambda=\{\lambda_0,\lambda_1,\ldots, \lambda_{\ell-1}\}$.  This is usually done 
when $W$ is the symmetry group of a regular polytope and $\Lambda$ is chosen 
from the vertices of the barycentric subdivision 
in such a way that each $\lambda_i$ corresponds to an $i$-dimensional face of the polytope; 
see Figures~\ref{Fig:Our},~\ref{Fig:CubeExample}, and Section~\ref{Background}.

Before stating the next definition, we make precise what is meant by 
(piecewise) $\mathbb R$-linearly extending a map 
$F:{\text{Vert}}(\Delta)\to V$ on vertices of a simplicial poset $\Delta$.  
Identifying $\Delta$ with a $CW$-complex as in~\cite{Bjorner:Top}, 
each point $y\in \Delta$ is contained in a unique (open) cell $\sigma$.  
Because $\Delta$ is a simplicial poset, the cell 
admits a characteristic map $f$ that maps the standard 
$(\dim \sigma)$-simplex onto $\overline{\sigma}$ 
while restricting 
to a bijection between vertices ($0$-cells).  
Thus, there are unique scalars $c_v$ so that
\[f^{-1}(y)=\sum_{v\in\text{Vert}(\overline{\sigma})} c_vf^{-1}(v),
\quad \sum_{v\in\text{Vert}(\overline{\sigma})} c_v =1,\quad\text{and}\quad c_v\in\mathbb R_{\geq 0}\text{ for all $v$}.\]
We extend $F$ to all of $\Delta$ by defining 
\[F(y)=\sum_{v\in\text{Vert}(\overline{\sigma})} c_v F(v).\]
We can now state the main definitions of this paper.

\begin{definition}\label{Def:Well-Framed}
Let $(W,R,\Lambda)$ be a framed system.  
Define $\rho:||\Delta(W,R)||\longrightarrow V$ by the following map on vertices
\begin{align*}
gW_{R\diff \{r_i\}}
&\stackrel{\rho}{\longmapsto}g\lambda_i
\intertext{(which is well-defined because the subgroup $W_{R\diff \{r_i\}}$ fixes $\lambda_i$)
and then extending $\mathbb R$-linearly over each face $gW_{R\diff J}$ of $\Delta$, so} 
 gW_{R\diff J}&\stackrel{\rho}{\longmapsto} g\Lambda_J,
\end{align*}
where 
\[\Lambda_J:=\Hull(\{\lambda_i\ :\ r_i\in J\}).\]
\begin{itemize}
\item We say that $(W,R,\Lambda)$ is \emph{well-framed} 
 if the equivariant map $\rho$ is an embedding.
\item If, in addition, each $X\in \L_W$ contains the image of at least 
one $(\dim X-1)$-face under 
$\rho$, then we say that 
 $(W,R,\Lambda)$ is \emph{strongly stratified}.
\end{itemize} 
\end{definition}

\begin{convention}\label{Convention}  
Given a well-framed system $(W,R,\Lambda)$, 
we will often identify the abstract simplicial poset $\Delta(W,R)$ and the 
geometric realization $\rho(\Delta(W,R))$ 
afforded by $\rho$.
\end{convention}

Observe that 
a system $(W,R,\Lambda)$ is well-framed if and only if for any positive real constants 
$c_1,c_2,\ldots,c_\ell$, the system $(W,R,\{c_1\lambda_1,c_2\lambda_2,\ldots,c_\ell\lambda_\ell\})$ is well-framed.

\begin{example}\label{Example:Coxeter}
Let $(W,R)$ be a simple system, i.e., an irreducible finite real reflection group $W$ with 
$R$ the set of reflections through 
walls of a Weyl chamber.  If $W$ is a Weyl group, 
one can obtain a real frame by choosing $\Lambda$ to be the set of fundamental dominant weights.  
More generally, a real frame is obtained by choosing one nonzero point on 
each extreme ray of a Weyl chamber corresponding to $R$.  
The resulting system $(W,R,\Lambda)$ is strongly stratified and 
$\rho(\Delta(W,R))$ is homeomorphic to the Coxeter complex via radial projection; Figure~\ref{Fig:Embedding} 
illustrates the construction for $W=I_2(5)$, the dihedral group of order 10.
\end{example}
\begin{center}
\begin{figure}[hbt]
\includegraphics{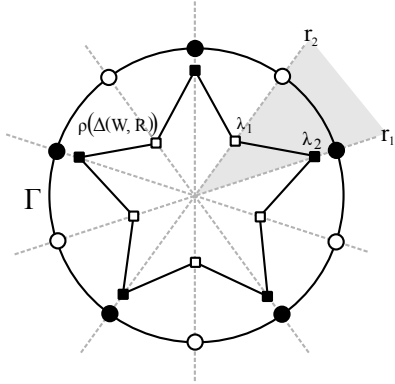}
\caption{A well-framed system and shaded Weyl chamber for $I_2(5)$.}\label{Fig:Embedding}
\end{figure}
\end{center}

\begin{definition}
We say that $(W,R)$ is \emph{well-framed} or \emph{strongly stratified} 
if there exists a frame $\Lambda$ for which $(W,R,\Lambda)$ is 
well-framed or strongly stratified, respectively.
\end{definition}

For $W$ a real reflection group, 
the following example suggests that only simple systems $(W,R)$ can be 
well-framed by a real frame $\Lambda\subset \mathbb R^\ell$.  In other words, the well-framed 
systems $(W,R,\Lambda)$ with $\Lambda$ real and $W$ a (complexified) 
finite real reflection group are completely characterized in Example~\ref{Example:Coxeter}; see 
Corollary~\ref{Corollary:Coxeter} below.

\begin{example}\label{NonEmbedding}  
Let $W=I_2(5)$, the dihedral group of order $10$, and let $R=\{r_1,r_2\}$, where $r_1,r_2$ 
are the reflections 
indicated in Figure~\ref{Fig:NonEmbedding}.  
In this case, $(W,R)$ is not a simple system, as the corresponding hyperplanes 
$H_1,H_2$ do not form a Weyl chamber; see the shaded region.    
Two real frames are shown in the figure, one with $|\lambda_1|\neq |\lambda_2|$ (left) and 
one with $|\lambda_1|=|\lambda_2|$ (right).  Both 
fail to yield a well-framed system.  For example, on the right in Figure~\ref{Fig:NonEmbedding}, when 
$|\lambda_1|=|\lambda_2|$, the map 
$||\Delta(W,R)||\stackrel{\rho}{\rightarrow}\rho(\Delta)$ is a double covering.
\begin{center}
\begin{figure}[hbt]
\includegraphics{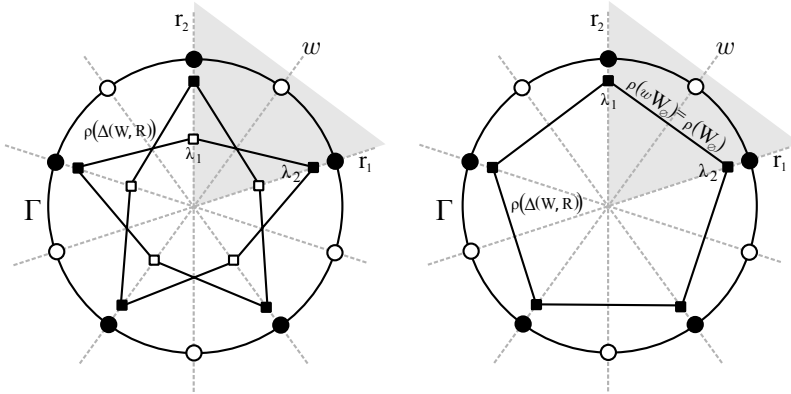}
\caption{Systems $(W,R,\Lambda)$ that are not well-framed for $W=I_2(5)$, $R=\{r_1,r_2\}$, and 
real $\Lambda=\{\lambda_1,\lambda_2\}$.}\label{Fig:NonEmbedding}
\end{figure}
\end{center}
\end{example}

Comparing Figures~\ref{Fig:Embedding} and~\ref{Fig:NonEmbedding}, we see that for a 
fixed $W$, it is possible to have both good and bad choices for $R$ and $\Lambda$.  Also 
observe that $(W,R,\Lambda)$ being well-framed is a global property, not a local one.  
For example, 
Figure~\ref{Fig:NonEmbedding} shows that for $\ell(w)$ large, 
one has to check for intersections of the simplices $\rho(eW_\varnothing)$ and $\rho(wW_\varnothing)$.

Though some systems $(W,R)$ do not give a well-framed triple $(W,R,\Lambda)$ for any \emph{real} $\Lambda\subset\mathbb R^\ell$, one may still 
be able to choose a good frame $\Lambda\subset \mathbb C^\ell$, as in the following example.

\begin{example}\label{TriangleExample}  
Let $W=I_2(3)$ and let $R=\{r_1,r_2\}$, where $r_1,r_2$ and $r_3$ are the reflections indicated in 
Figure~\ref{Fig:TriangleExample}.  
As in Example~\ref{NonEmbedding}, $(W,R)$ is not a simple system because 
the corresponding hyperplanes $H_1,H_2$ 
do not form a Weyl chamber.  If $\lambda_1,\lambda_2$ are real and as shown in 
Figure~\ref{Fig:TriangleExample}, the triple $(W,R,\Lambda)$ is not well-framed.  In fact, for 
every choice of a real frame $\Lambda\subset\mathbb R^2$, the resulting system $(W,R,\Lambda)$ 
is not well-framed.
However, it is still possible to construct a well-framed system from $(W,R)$ using 
$\Lambda\subset \mathbb C^2$.

Let $H_1,H_2,H_3$ be the (complex) reflecting hyperplanes for $r_1,r_2,r_3$, respectively, and 
choose coordinates so that
\[H_1=\mathbb C\left[\begin{matrix}\sqrt{3}\\ -1\end{matrix}\right],\quad
 H_2=\mathbb C\left[\begin{matrix}0\\ 1\end{matrix}\right],\quad
H_3=\mathbb C\left[\begin{matrix}\sqrt{3}\\ 1\end{matrix}\right].\]

Let
\[\lambda_2=\frac{1}{4}\left[\begin{matrix}\sqrt{3}\\-1\end{matrix}\right]
\quad\text{and}\quad
\lambda_1=i\left[\begin{matrix} 0 \\ 1\end{matrix}\right]
.\]
Then it is straightforward to verify that $(W,R,\Lambda)$ is a well-framed system.  
For example, the two segments 
\begin{align*}
\sigma_1&=\Lambda_R=\left\{ti\left[\begin{matrix} 0 \\ 1\end{matrix}\right]
+(1-t)\frac{1}{4}\left[\begin{matrix}\sqrt{3}\\ -1\end{matrix}\right]\ :\ 0\leq t\leq 1\right\}\\
\intertext{and}
\sigma_2&=w\sigma_1=\left\{s\frac{i}{2}\left[\begin{matrix} \sqrt{3} \\ -1\end{matrix}\right]
+(1-s)\frac{1}{2}\left[\begin{matrix} 0 \\ 1\end{matrix}\right]
\ :\ 0\leq s\leq 1\right\}
\end{align*}
do not intersect, since they 
have distinct endpoints and there is no real solution to 
\[(1-t)\frac{\sqrt{3}}{4}=si\frac{\sqrt{3}}{2}\qquad 0< s,t< 1.\]

Note also that the linear form 
$\alpha_{H_3}:=x_1-\sqrt{3}x_2$
is nonzero on $\sigma_1$ and $\sigma_2$, implying that $H_3$ neither intersect $\sigma_1$ nor 
$\sigma_2$.
\begin{center}
\begin{figure}[hbt]
\includegraphics{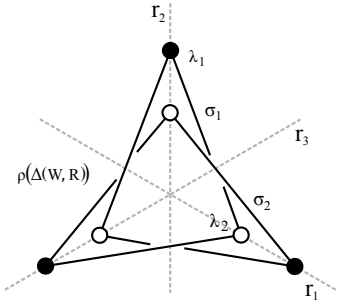}
\caption{System $(W,R,\Lambda)$ with $W=I_2(3)$ and 
$R=\{r_1,r_2\}$.  The system $(W,R)$ is well-framed for some nonreal 
$\Lambda=\{\lambda_1,\lambda_2\}\subset \mathbb C$, but 
not for any real $\Lambda\subset \mathbb R$.}\label{Fig:TriangleExample}
\end{figure}
\end{center}
\end{example} 

\section{Geometry and algebra for frames}\label{geometry}

Though $\Delta(W,R)$ is generally only a Boolean complex, existence of a well-framed system 
$(W,R,\Lambda)$ forces it to be a simplicial complex:

\begin{proposition}\label{Proposition:Simplicial}
If $(W,R,\Lambda)$ is well-framed, then $\Delta:=\Delta(W,R)$ is a balanced simplicial complex.
\end{proposition}

\begin{proof}
The image of a face under $\rho$ is determined by its vertices.  Because 
$\rho$ is assumed to be an embedding, it follows that any two faces of $\Delta$ with 
equal vertex sets must be the same face.
\end{proof}

As a corollary, we see that every well-framed system $(W,R,\Lambda)$ with $W$ a real 
reflection group and $\Lambda\subset \mathbb R^\ell$ is of the type constructed in Example~\ref{Example:Coxeter}.  
We make this precise in the following corollary, whose proof is straightforward and left to the reader.

\begin{corollary}\label{Corollary:Coxeter}
Let $(W,R)$ be a well-generated system with $W$ a (finite) real reflection group.  Assume 
that $\Lambda\subset \mathbb R^\ell$ is a real frame.  
Then $(W,R,\Lambda)$ is well-framed if and only if 
\[\mathbb R_{\geq 0}\lambda_1,\ldots,\mathbb R_{\geq 0}\lambda_\ell\]
are the extreme rays of a Weyl chamber. 
Moreover, if $(W,R,\Lambda)$ is well-framed, then $(W,R)$ is a simple system.
\end{corollary}

We now come to the notion of \emph{support}.  Note that in the following definition, 
we could very well replace $\Supp$ with $\LinSpan$, given Convention~\ref{Convention}.  However, 
it will be helpful to distinguish between the two.

\begin{definition}
Let $(W,R,\Lambda)$ be well-framed.  We define the \emph{support} map
\begin{alignat*}{1}
\Supp :  \Delta  &\longrightarrow \L(V)\\
gW_J  &\longmapsto  \LinSpan(\rho(gW_J)),
\end{alignat*}
and let
\[\Pi_W\Def\{\Supp(\sigma)\mid\sigma\in \Delta\},\]
viewed as a subposet of $\L(V)$. 
\end{definition}

As for $\L_W$, we will often write $\Pi$ in place of $\Pi_W$.  
We start by observing that for a well-framed system, $\Supp:\Delta_W\to \L_W$.

\begin{proposition}\label{Pi:L}
For a well framed system $(W,R,\Lambda)$, we have 
\[\Pi_W\subseteq \L_W,\]
with equality if and only if $(W,R,\Lambda)$ is strongly stratified.
\end{proposition}

\begin{proof}
Consider a coset $gW_{R\diff J}$, and recall that 
$\rho(gW_{R\diff J})=g\Lambda_J$.
Using the definition of $\Lambda$, we also have that
\[
\LinSpan(\Lambda_J)=\bigcap_{r_i\in R\diff J} H_i.
\]
The inclusion follows.   The second claim follows from the definition of a strongly 
stratified system by considering dimension.
\end{proof}

The main theorem of this section is that the equivariant support map 
for a well-framed system has a purely algebraic description.  The 
mechanism enabling this characterization is the \emph{Galois correspondence}, 
an analogue of Barcelo and Ihrig's Galois correspondence; 
see~\cite{BarceloIhrig}, 
where they utilize the Tits cone to establish a result generalizing the real case 
to all Coxeter groups.  
In order to state the correspondence, we introduce the poset of \emph{standard 
parabolics}
\[P(W,R)\Def\{gW_Jg^{-1}\ :\ g\in W, J\subseteq R\},\]
ordered by inclusion, i.e., 
\[gW_Jg^{-1}<hW_{J'}h^{-1}\quad \text{if}\quad gW_Jg^{-1}\subset hW_{J'}h^{-1}.\]
Note that $W$ acts on $P(W,R)$ via conjugation.

\begin{theorem}[Galois Correspondence]\label{Galois}
For $(W,R,\Lambda)$ a well-framed system,
\begin{alignat*}{2}
\Stab :\ \Pi_W  &\longrightarrow\ & &\Par \\
        X   &\longmapsto\     & &\{g\in W\ :\ gx=x\ \text{for all }x\in X\}
\end{alignat*}
is an $W$-poset isomorphism, with inverse
\begin{alignat*}{2}
\Fix :\ \Par  &\longrightarrow\ & &\Pi_W \\
        G     &\longmapsto\ & &V^G:=\{v\in V\ :\ gv=v\quad\text{for all}\quad g\in G\}.
\end{alignat*}
\end{theorem}

\begin{proof}
Using the fact that $\rho$ is an equivariant embedding of a balanced complex, 
\begin{align*}
\Stab(\Supp(gW_{R\diff J})) &= \{w\in W\ :\ w\rho(gW_{R\diff S})=
\rho(gW_{R\diff S})\ \text{pointwise}\}\\
&=\{w\in W\ :\ wgW_{R\diff S}=gW_{R\diff S}\}\\
&=gW_{R\diff S}g^{-1}.
\end{align*}

Regarding the inverse, we have
\[
V^{g(W_{R\diff S})g^{-1}}=gV^{W_{R\diff S}}=g\bigcap_{r_i\in R\diff S}H_i
=g\cdot\LinSpan(\Lambda_S)
=\Supp(gW_{R\diff S}).
\]
\end{proof}

The promised algebraic interpretation of support is now encoded in an 
equivariant commutative diagram:

\begin{theorem}\label{Commutative}
For a well-framed system $(W,R,\Lambda)$, we have the following commutative 
diagram of equivariant maps:
 \[\xymatrixcolsep{.5cm}\xymatrixrowsep{3pc}
 \xymatrix{
 gW_J\ar@{|->}[d] & \Delta\ar@{.>>}[d]^-{}_-{}
\ar@{->}[rrr]^-{\sim}_-{\rho} &  & &
\rho(\Delta) \ar@{->>}[d]^-{\LinSpan}_-{}
&  \\
gW_Jg^{-1} &{\Par}\ar@<.8ex>[rrr]^-\Fix_-{} &  & & 
\Pi_W \ar@<.8ex>[lll]^-\Stab_-\sim \ar@<-.4ex>@{^{(}->}[r]^-{\iota}_-{}&  \L_W
 }
 \]
If $(W,R)$ is strongly stratified, then $\Pi_W=\L_W$.
\end{theorem}

\section{Pointed objects}\label{Section:Local}
This section introduces the main objects of this paper, the generalizations of 
Ehrenborg and Jung's pointed objects.  
Recall that the \emph{(closed) star} $\St_\Sigma(\sigma)$ of a simplex 
$\sigma$ in a simplicial complex $\Sigma$ is the 
subcomplex of all faces that are joinable to $\sigma$ within $\Sigma$.  That is,  
\[\St_\Sigma(\sigma)=\{\tau\ :\ \tau\in \Sigma\quad
\text{and}\quad \tau\cup\sigma\in \Sigma\}.\]
We will write $\tau*\sigma$ for the join $\tau\cup\sigma$.

\begin{definition}\label{PointedObjects}
Let $(W,R,\Lambda)$ be a well-framed system.  Let $U\subseteq R$.  
The \emph{subcomplex pointed by $U$} is 
\begin{align*}
\Delta^U\ &\Def\ \St_\Delta(W_{R\diff U}).\\
\intertext{The corresponding \emph{pointed poset of flats} is}
\Pi^U\ &\Def\ \{\Supp(\sigma)\ :\ \sigma\in\Delta^U\},
\end{align*}
as a subposet of $\L_W$.  

For $T\subseteq R$, let 
\[\Delta_T^U\ \Def\ \Delta^U\mid_T\quad\text{and}\quad
\Pi_T^U\ \Def\ \{\Supp(\sigma)\ :\ \sigma\in\Delta_T^U\}.\]
\end{definition}

Note that by Theorem~\ref{Commutative}, we have $\Supp(gW_J)=V^{gW_Jg^{-1}}$, showing that
\begin{equation}\label{Pi:Rewrite}
\Pi_T^U\ = \{V^{gW_Jg^{-1}}\ :\ gW_J\in\Delta_T^U\}
\end{equation}

Figure~\ref{Fig:Our} illustrates the construction of 
$\Delta_T^U$ for $W=\mathfrak S_{4}$ and eight choices of $T$ and $U$.  We have 
written ``$U$'' above each element of $U$, and ``$T$'' below each element of $T$ in the Coxeter 
diagram $\mathcal D$ of $(W,R)$.  Labeling the generators $R=\{r_0,\ldots, r_{\ell-1}\}$, a vertex 
marked $i$ in $\mathcal D$ represents the reflection $r_i$ whose hyperplane can be 
written as
\[H_i=\LinSpan(\Lambda\diff\{\lambda_i\}).\]

Recall that $W$ is the symmetry group of the tetrahedron $\P$, so the 
barycentric subdivision $B(\P)$ of the boundary of $\P$ 
is homeomorphic to the Coxeter complex via radial projection.
The vertex of $\Delta$ marked with $i$ corresponds to $\lambda_i$.  It 
happens that $W$ is also a Shephard group, because $\P$ is a regular polytope.  
In the later notation of Shephard groups, vertex $i$ will correspond to a face $B_i$ in a 
distinguished \emph{base flag} $\calB$ (a chamber in the flag complex $K(\P)$ for 
the polytope $\P$).  

\begin{figure}[hbt]
\begin{center}
{\psfrag{0}{\tiny{0}}
\psfrag{1}{\tiny{1}}
\psfrag{2}{\tiny{2}}
\psfrag{U}{\tiny{$P$}}
\psfrag{T}{\tiny{$T$}}
\includegraphics{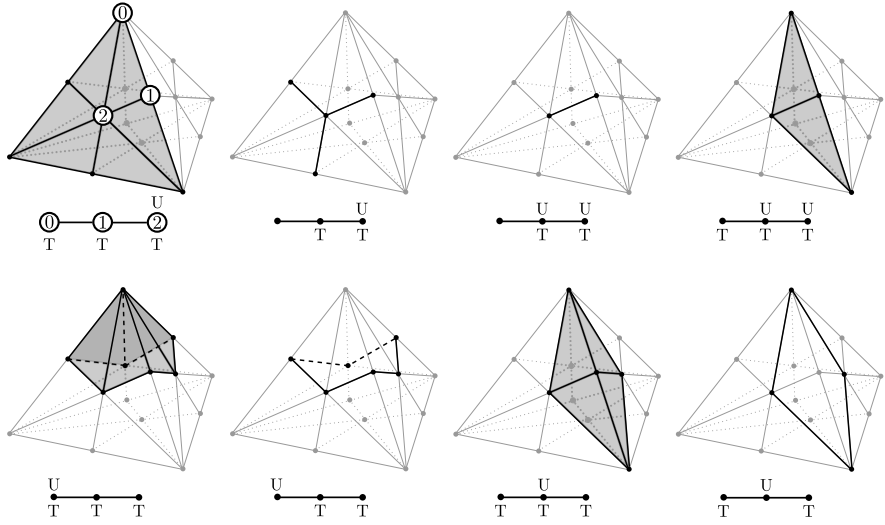}
}
\caption{A sampling of $\Delta_T^U$ for $W=\mathfrak S_4$.}\label{Fig:Our}
\end{center}
\end{figure}

Similarly, Figure~\ref{Fig:CubeExample} illustrates the construction for the hyperoctahedral group $\mathbb Z_2\wr\mathfrak S_n$ of 
$n\times n$ signed permutation matrices.  
Recall that this is the symmetry group of the $n$-cube $\P$, so its Coxeter complex 
is a radial projection of the barycentric subdivision $B(\P)$ of the boundary of $\P$.  
Again, we let $i$ 
denote our choice of $\lambda_i$, and in later notation, a face $B_i$ of a distinguished 
base flag $B_0\subset\cdots\subset B_{\ell-1}$ in the flag complex $K(\P)$ of $\P$.

\begin{figure}[hbt]
\begin{center}
\includegraphics{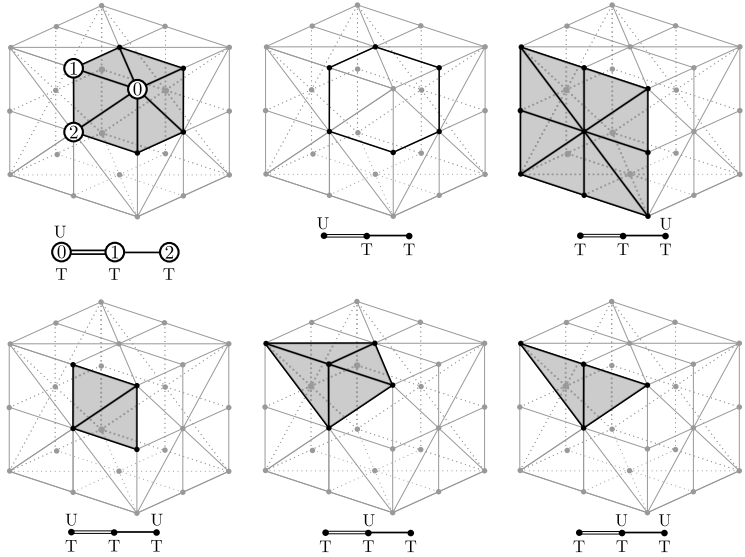}
\caption{A sampling of $\Delta_T^U$ for $W=\mathbb Z_2\wr \mathfrak S_3$.}\label{Fig:CubeExample}
\end{center}
\end{figure}

\section{Equivariant homotopy for locally conical systems}\label{Section:Equivariant}

Recall that a map $f:P\to Q$ of posets is \emph{order-preserving} 
if $f(p_1)\leq f(p_2)$ whenever $p_1\leq p_2$, and \emph{order-reversing} if 
$f(p_1)\geq f(p_2)$ whenever $p_1\leq p_2$.  A $G$-poset is a poset with a $G$-action that 
preserves order, and a map $f:P\to Q$ of such posets is \emph{$G$-equivariant} 
if it is a mapping of $G$-sets, i.e., if $f(gp)=gf(p)$ for all $g\in G$ and $p\in P$.  

The order complex of a poset $P$, i.e., the simplicial complex of all totally ordered 
subsets of $P$, is denoted $\Delta(P)$.  The \emph{face poset} $\Face(\Sigma)$ 
of a simplicial complex $\Sigma$ is the poset of all \emph{nonempty} faces ordered by inclusion.  
Finally, $\homotopic$ denotes homotopy equivalence, 
with added decoration to indicate equivariance.  
Note that the barycentric subdivision $\Delta(\Face(\Sigma))$ is homeomorphic to $\Sigma$.

The aim of this section is to establish a sufficient condition for the order complex of 
$\Pi_T^U \diff \{\hat{1}\}$ to be equivariantly homotopy equivalent to $\Delta_T^U$. 
Our main tool is a specialization\footnote{The main theorem of~\cite{Webb} states that 
$\Delta(P)\homotopic_G\Delta(Q)$ if $\phi:P\to Q$ is an order-preserving 
$G$-equivariant map of $G$-posets such that each fiber $\Delta(\phi^{-1}(Q_{\leq q}))$ 
is $\Stab_G(q)$-contractible.  Theorem~\ref{Webb} 
specializes $P$ to $\Face(\Sigma)$ and replaces $Q$ with its \emph{opposite}
$Q^{\rm{opp}}$, using the fact that $\Sigma\homotopic_G \Delta(\Face(\Sigma))$ and 
$\Delta(Q^{\rm{opp}})=\Delta(Q)$.  Recall that $Q^{\rm{opp}}$ is obtained from 
$Q$ by reversing order.} of Th\'evenaz and Webb's equivariant version of Quillen's fiber lemma.

\begin{theorem}[Th\'evenaz and Webb~\cite{Webb}]\label{Webb}
Let $Q$ be a $G$-poset, and let $\Sigma$ be a simplicial complex with a $G$-action.  If 
$\phi:\Face(\Sigma)\to Q$ is an order-reversing $G$-equivariant map of $G$-posets such that 
the order complex $\Delta(\phi^{-1}(Q_{\geq q}))$ is $\Stab_G(q)$-contractible for all $q\in Q$, 
then $\Sigma\homotopic_G\Delta(Q)$.  
\end{theorem}

Consider a simplicial complex $\Sigma$ and order-reversing map
$\phi:\Face(\Sigma)\to Q$ of posets.  
Since $\phi^{-1}(Q_{\geq q})$ is an order ideal in $\Face(\Sigma)$, it is the 
face poset $\Face(\Phi)$ of some subcomplex $\Phi\subseteq \Sigma$.  
Call such a subcomplex $\Phi$ a \emph{Quillen fiber}.  Note that 
$\Delta(\phi^{-1}(Q_{\geq q}))$ is homeomorphic to $\Phi$, so Quillen's 
fiber lemma concerns contractibility of Quillen fibers.

The following definition of a \emph{locally conical system} is central to our work.  
In particular, we will show that for such systems, Theorem~\ref{Webb} can be applied to 
establish the desired homotopy equivalence.

\begin{definition}
A well-framed system $(W,R,\Lambda)$ is \emph{locally conical} if for each nonempty 
$U\subseteq R$, every Quillen fiber 
\[X\cap \Delta^U\qquad(X\in\Pi^U\diff\{\hat{1}\})\] of 
$\Supp:\Face(\Delta^U)\to \Pi^U\diff\{\hat{1}\}$ 
has a cone point.
\end{definition}

Recall that a \emph{cone point} $p$ of a simplicial complex $\Sigma$ is a vertex 
of $\Sigma$ that is joinable in $\Sigma$ to every simplex of $\Sigma$, i.e., every 
maximal simplex of $\Sigma$ contains $p$.

\begin{proposition}\label{Prop:Contraction}
Let $\Sigma$ be a simplicial complex with a $G$-action.  
If $\Sigma$ has a cone point, then $\Sigma$ is $G$-contractible.
\end{proposition}

\begin{proof}
The union of all cone points must form a $G$-stable simplex of $\Sigma$, 
whose barycenter $p$ is therefore a $G$-fixed point of the geometric realization $\|\Sigma\|$. 
Since $p$ lies in a common simplex with every
simplex of $\Sigma$, this space $\|\Sigma\|$ is star-shaped with respect to the $G$-fixed point 
$p$, and a straight-line homotopy retracts $\|\Sigma\|$ onto $p$ in a $G$-equivariant fashion.
\end{proof}

Before employing Theorem~\ref{Webb}, recall 
that $\Delta^U=\St_\Delta(W_{R\diff U})$ and that the action of $W$ on $\Delta$ preserves 
types.  Hence, the subcomplex $\Delta^U_T$ is a $W_{R\diff U}$-poset.   It follows that
\[\Supp:\Face(\Delta^U_T)\to \Pi^U_T\diff\{\hat{1}\}\]
is an order-reversing $W_{R\diff U}$-equivariant map.

\begin{theorem}\label{Equivalence}
Let $(W,R,\Lambda)$ be a locally conical system. 
Let $U\subseteq R$ be nonempty and $T\subseteq R$.  Then 
$\Delta(\Pi_T^U\diff\{\hat{1}\})$ is $W_{R\diff U}$-homotopy equivalent 
to $\Delta_T^U$.
\end{theorem}

\begin{proof}
We apply Theorem~\ref{Webb} to the map 
\[
\Supp: \Face(\Delta_T^U)\twoheadrightarrow \Pi_T^U\diff\{\hat{1}\}.
\] 
Let $X\in\Pi_T^U\diff\{\hat{1}\}\subseteq \Pi^U\diff\{\hat{1}\}$, and 
consider first the Quillen fiber
\begin{equation}\label{UnrestrictedFiber}
\Phi:=X\cap \Delta^U
\end{equation}
for the unrestricted map $\Supp:\Face(\Delta^U)\to \Pi^U\diff\{\hat{1}\}$.

By definition of locally conical system, 
$\Phi$ has a cone point $p$.  Since $\Delta$ is balanced, the subcomplex 
$\Phi$ is also balanced.  It follows that $p$ is the unique vertex of $\Phi$ of 
its type.  

Choose $\sigma\in\Delta_T^U$ with $X=\Supp(\sigma)$.  
By the construction of $\rho(\Delta)$, the vertices of the join $\{p\}*\sigma$ 
are contained in 
\[g\Lambda=\{g\lambda_1,g\lambda_2,\ldots, g\lambda_\ell\}\]
for some $g\in W$.  Because $g\Lambda$ is a linearly independent set,
$p\in\Supp(\sigma)$ implies that $p$ is a vertex of $\sigma$, and hence 
a vertex of $\Delta_T^U$.  
Therefore, the restricted Quillen fiber $X\cap \Delta_T^U$ 
also has $p$ as a cone point.  
It follows from Proposition~\ref{Prop:Contraction} 
that $X\cap\Delta_T^U$ is $\Stab_{W_{R\diff U}}(X)$-contractible.
\end{proof}

\section{Homology of locally conical systems}
\label{Section:Homology}

By applying the homology functor to Theorem~\ref{Equivalence}, we have the following

\begin{theorem}\label{Theorem:Main}  
Let $(W,R,\Lambda)$ be a locally conical system.  Let $U\subseteq R$ be 
nonempty and $T\subseteq R$.  Then
$\bigoplus_i\tilde{H}_{i}(\Delta_T^U)$
and 
$\bigoplus_i\tilde{H}_{i}(\Delta(\Pi_T^U\diff \{\hat{1}\}))$
are isomorphic as 
graded $(W_{R\diff U})$-modules.
\end{theorem}

Recall that a simplicial complex $\Sigma$ is 
\emph{Cohen-Macaulay} (over $\mathbb Z$) if 
for each $\sigma\in \Sigma$ we have $\tilde{H}_i(\lk\ \sigma,\mathbb Z)=0$ whenever 
$i<\dim(\lk_\Sigma\ \sigma)$, where 
$\lk_\Sigma$ denotes the \emph{link}:
\[\lk_\Sigma(\sigma)=\{\tau\in \Sigma\ :\ \tau\cap \sigma=\varnothing,\ 
\tau\cup \sigma\in 
\Sigma\}.\]
The complex $\Sigma$ is  \emph{homotopy Cohen-Macaulay} if 
for each $\sigma\in \Sigma$ we have $\pi_r(\lk_\Sigma(\sigma))=0$ whenever $r\leq \dim\lk_\Sigma(\sigma)-1$.  
Homotopy Cohen-Macaulay implies Cohen-Macaulay (over $\mathbb Z$), 
and Cohen-Macaulay implies 
Cohen-Macaulay over any field; see the appendix of~\cite{Bjorner}.

This section is devoted to establishing explicit descriptions for the modules in 
Theorem~\ref{Theorem:Main} 
when $\Delta$ is Cohen-Macaulay.  We start with the following

\begin{proposition}\label{CM:Typed}
Let $(W,R,\Lambda)$ be a well-framed system.  Let $U,T\subseteq R$.   If $\Delta$ is 
Cohen-Macaulay (resp. homotopy Cohen-Macaulay), then $\Delta_T^U$ is Cohen-Macaulay 
(resp. homotopy Cohen-Macaulay).  
\end{proposition}

\begin{proof} 
It is an easy exercise to show that stars inherit the Cohen-Macaulay 
(resp. homotopy Cohen-Macaulay) property.  
The nontrivial step is concluding that $\Delta_T^U$ is Cohen-Macaulay 
(resp. homotopy Cohen-Macaulay).  This follows 
from a type-selection theorem for pure simplicial complexes; see 
Bj\"orner~\cite[Thm. 11.13]{Bjorner:Top} and Bj\"orner, Wachs, and Welker~\cite{BWW}.
\end{proof}

\begin{lemma}\label{Lemma:Contractible}  
Let $(W,R,\Lambda)$ be a well-framed system.  
If $U\cap T\neq\varnothing$, then $\Delta_T^U$ is contractible.  
\end{lemma}

\begin{proof}
This is immediate from $\Delta^U=\{W_{R\diff U}\}*\lk_{\Delta}W_{R\diff U}$.
\end{proof}

\begin{theorem}\label{Homology}
Let $(W,R,\Lambda)$ be a well-framed system.  Let $U,T\subseteq R$, and assume that $\Delta$ is Cohen-Macaulay.  
If $U\cap T\neq\varnothing$, then the top homology 
$\tilde{H}_{|T|-1}(\Delta_T^U)$ is trivial; otherwise, 
\begin{equation}\label{Modules}\tilde{H}_{|T|-1}(\Delta_T^U)
\cong \sum_{J\subseteq T} (-1)^{|T\diff J|}\Ind_{W_{ R\diff (U\cup J)}}^{W_{R\diff U}} \triv
\end{equation}
as virtual $(W_{R\diff U})$-modules.
\end{theorem}

\begin{remark}
Specializing to type $A$ and to Ehrenborg and Jung's objects (see 
Section~\ref{Section:Ehrenborg}), Lemma~\ref{Lemma:Contractible} 
translates to
$\Delta_{\vec{c}}$ being contractible whenever $\vec{c}$ ends with a zero; 
this is precisely Lemma 3.1 in~\cite{Ehrenborg}.  The condition that 
$U\cap T= \varnothing$ collapses to the condition that 
$\vec{c}$ not end with zero.

We also note that 
the virtual modules in~\eqref{Modules} are well-known to be the natural 
generalization of ribbon representations to all Coxeter and Shephard groups; see~\cite{Solomon:Decomp} and~\cite{OrlikReinerShepler}.
\end{remark}

Though the proof of Theorem~\ref{Homology} is entirely standard, we first need a particular 
description of $\Delta_T^U$ when $U\cap T=\varnothing$.  
The following intermediate description is straightforward.

\begin{lemma}\label{Parabolic:Typed}  
Let $(W,R)$ be a well-generated system.  Then 
for $U,T\subseteq R$ we have
\begin{equation}\label{Delta:description}
\Delta_T^U=\{gW_{R\diff J}\ :\ g\in W_{R\diff U},\ J\subseteq T\}.
\end{equation}
\end{lemma}

When $\Delta(W,R)$ is a simplicial complex, one has that $(W,R)$ satisfies 
the \emph{intersection condition} 
\[\bigcap_{r\in R\diff J} W_{R\diff \{r\}}=W_{J}\quad\text{for all $J\subseteq R$}.\]
In fact, satisfying the intersection condition is equivalent to $\Delta(W,R)$ being 
a simplicial complex; see~\cite[Cor. 2.6]{BabsonReiner}.  The following lemma 
is a straightforward application of the intersection property, and the desired 
description of $\Delta_T^U$ is obtained by applying the lemma to~\eqref{Delta:description}.

\begin{lemma}\label{Induced}  
Let $(W,R)$ be well-framed.  If $U\cap T=\varnothing$, then the map
\begin{align*}
\{gW_{R\diff J}\ :\ g\in W_{R\diff U},\ J\subseteq T\}
&\longrightarrow 
\{gW_{R\diff (U\cup J)}\ :\ g\in W_{R\diff U},\ J\subseteq T\}\\
gW_{R\diff J}&\longmapsto gW_{R\diff (U\cup J)}
\end{align*}
is a $(W_{R\diff U})$-poset isomorphism.
\end{lemma}

The proof of Theorem~\ref{Homology} now follows:

\begin{proof}[Proof of Theorem~\ref{Homology}]  
The first claim follows from Lemma~\ref{Lemma:Contractible}.  When $U\cap T=\varnothing$, 
the result follows from the description of $\Delta_T^U$ 
obtained through Lemmas~\ref{Parabolic:Typed} and~\ref{Induced} 
by applying the standard argument using Cohen-Macaulayness and the Hopf 
trace formula, as is detailed by Solomon in~\cite{Solomon:CM}.  Other good sources 
include~\cite{Stanley:Aspects} and the notes of Wachs~\cite{Wachs:Tools}.
\end{proof}

\section{Specializing to objects of  Ehrenborg and Jung}\label{Section:Ehrenborg}
In the case of type $A$, Ehrenborg and Jung constructed \emph{pointed objects} 
$\Delta_n,\Pi^\bullet_{n}$ and subcomplexes $\Delta_{\vec{c}},\Pi^\bullet_{\vec{c}}$ 
indexed by compositions $\vec{c}=(c_1,c_2,\ldots, c_k)$ of $n$ that have 
positive entries $c_1,\ldots, c_{k-1}>0$ but \emph{nonnegative} last entry $c_k\geq 0$.  
Figure~\ref{Fig:Ehrenborg:S3:Poset} shows Ehrenborg and Jung's $\Delta_2$ and 
$\Pi^\bullet_2$, each carrying an action of $\mathfrak S_2$.  In~\cite{Ehrenborg}, they show 
that the top homology of $\Delta_{\vec{c}}$ is homotopy equivalent to 
the order complex of $\Pi^\bullet_{\vec{c}}\diff\{\hat{1}\}$, and that the top homology is given by 
a ribbon Specht module of $\mathfrak S_n$ with row sizes determined by $\vec{c}$.  

\begin{figure}[hbt]
\begin{center}
{
\includegraphics[scale=2.20]{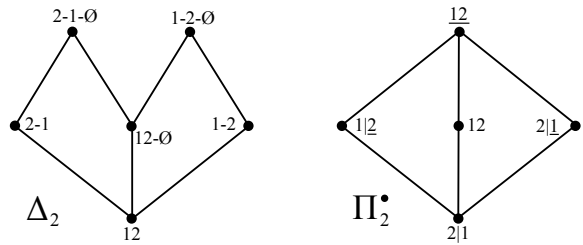}
}
\caption{Ehrenborg and Jung's pointed objects.}
\label{Fig:Ehrenborg:S3:Poset}
\end{center}
\end{figure}

The aim of this section is to present the underlying geometry of Ehrenborg and Jung's objects 
by showing how our objects $\Delta_T^U,\Pi_T^U$ specialize to theirs.  It will 
follow that applying our results to this specialization recovers their 
main results, even upgrading their homotopy equivalence to an equivariant one.  
The translations between objects of this section are well-known, and we will largely 
follow the discussion of Aguiar and Mahajan; see~\cite[\S 1.4]{Aguiar}.

Let $W=\mathfrak S_{n+1}$ and let $R=\{r_1,r_2,\ldots, r_n\}$ be the usual 
generating set of adjacent transpositions, i.e., 
\[r_i=(i,i+1).\]
The symmetric group $\mathfrak S_{n+1}$ is the symmetry group 
of the standard $n$-simplex $\P_n$ with vertices labeled by $\{1,2,\ldots, n+1\}$.  The 
barycentric subdivision $B(\P_n)$ of the boundary of $\P_n$ 
is homeomorphic to the Coxeter complex via 
radial projection.  Letting $\lambda_i$ be the vertex of $B(\P_n)$ indexed 
by $\{1,2,\ldots, i\}$ yields a well-framed system $(W,R,\Lambda)$ with 
$\rho(\Delta)=B(\P_n)$.  Note that, in particular, 
$\Lambda$ lies on the \emph{distinguished face} 
$\P_{n-1}$ of $\P_n$ that has vertex set $\{1,2,\ldots, n\}$; see Figure~\ref{E:intro}.  
We call $(W,R,\Lambda)$ the \emph{standard system for $\mathfrak S_{n+1}$}.

\begin{figure}[hbt]
\begin{center}
{
\includegraphics[scale=2.50]{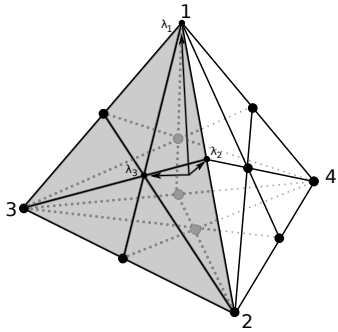}
}
\caption{$W=\mathfrak S_{4}$ with the subdivision of facet $\P_2$ shaded.}
\label{E:intro}
\end{center}
\end{figure}

A \emph{set composition} of $n+1$ is an ordered partition 
$B_1 \Dash B_2\Dash \cdots \Dash B_k$ of 
$[n+1]$ with nonempty blocks.  
The collection of all set compositions of $n+1$ form a simplicial complex $\Sigma_{n+1}$ under 
refinement, with chambers having 
$n+1$ singleton blocks; see Figure~\ref{Fig:Ehrenborg:S3}.  The \emph{type} of a set composition is the composition of its block sizes, i.e., 
\[\type(B_1 \Dash B_2\Dash \cdots \Dash B_k)=(|B_1|,|B_2|,\ldots,|B_k|).\] 
Given a composition $\vec{c}$ of $n+1$, the subcomplex of $\Sigma_{n+1}$ generated by 
the faces of type $\vec{c}$ is denoted $\Sigma_{\vec{c}}$.

The map $\phi$ obtained by letting
\[\{\lambda_{i_1},\lambda_{i_2},\ldots, \lambda_{i_k}\}\mapsto 1,\ldots, i_1 \Dash (i_1+1),\ldots, i_2\Dash\cdots\Dash(i_{k}+1),\ldots, n+1,\]
and extending by the action of $W$, is an equivariant isomorphism 
$\rho(\Delta)\rightarrow \Sigma_{n+1}$; see Figure~\ref{Fig:Ehrenborg:Map}. 
 Under this isomorphism, a face $B_1\Dash B_2\Dash \cdots\Dash B_k$ of type $\vec{c}$ 
corresponds to a face of type
\[{\rm{Des}}(\vec{c}):=\{r_{|B_1|},r_{|B_1|+|B_2|},\ldots, r_{|B_1|+\cdots+|B_{k-1}|}\}.\]
Note also that  those faces of $B(\P_{n})$ in the star of $\P_{n-1}$ either have $n+1$ 
contained in the last block or have last block equal to $\{n+1\}$.

\begin{figure}[hbt]
\begin{center}
{
\includegraphics[scale=2.50]{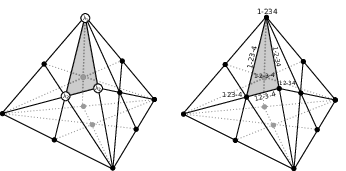}
}
\caption{}
\label{Fig:Ehrenborg:Map}
\end{center}
\end{figure}

The lattice of hyperplane intersections $\L_W$ for $W$ also has a simple 
description obtained from set compositions.  
Under the identification of $\rho(\Delta)$ and $\Sigma_{n+1}$, the support map 
corresponds to forgetting the order on the blocks.  That is,
\[{\rm Supp}(B_1\Dash B_2\Dash \cdots\Dash B_k)=B_1|B_2|\cdots|B_k.\]
The induced partial order is given by refinement.

A \emph{pointed set composition} of $n$ is an ordered partition 
$B_1\Dash B_2\Dash \cdots\Dash B_k$ of 
$[n]$ with last block $B_k$ possibly empty.  
We denote the collection of all pointed set compositions of $[n]$ by 
$\Delta_n^\bullet$.  The previous discussion shows that by removing $n+1$ from blocks in 
elements of $\Sigma_{n+1}$, the barycentric subdivision of the facet $\P_{n-1}$ 
can be identified with $\Delta_n^\bullet$; see Figures~\ref{Fig:Ehrenborg:Map} 
and~\ref{Fig:Ehrenborg}.  More accurately, 
$\Delta_n^\bullet$ is $\mathfrak S_n$-equivariantly isomorphic to $\Delta^{\{r_n\}}$.  
Given a composition $\vec{c}$ of $n$ with last part possibly $0$, the corresponding selected 
complex is denoted $\Delta_{\vec{c}}^\bullet$, which is Ehrenborg and Jung's complex 
$\Delta_{\vec{c}}$; see Figure~\ref{Fig:Ehrenborg}.  By distinguishing terminal blocks before 
taking the image of $\Delta_{\vec{c}}$ under the support map, one obtains their pointed poset 
$\Pi_{\vec{c}}^\bullet$ after removing any (possibly distinguished) empty blocks; see 
Figure~\ref{Fig:Ehrenborg:S3}.  

\begin{figure}[hbt]
\begin{center}
{
\includegraphics[scale=2.20]{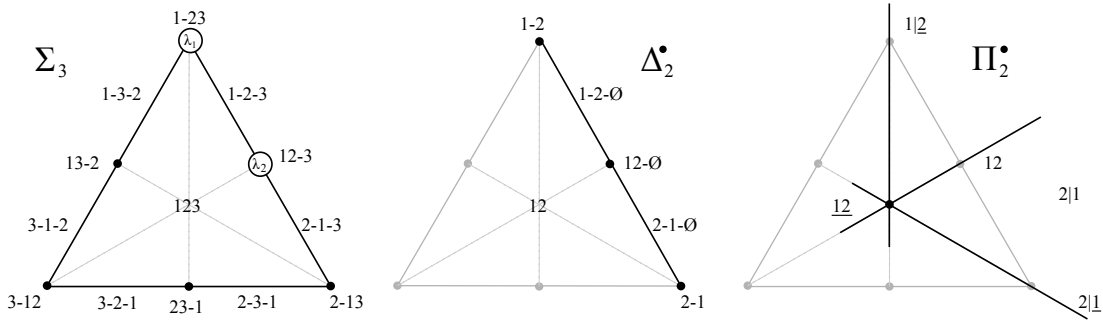}
}
\caption{The geometry of Ehrenborg and Jung's $\Delta_{2}$ and $\Pi^\bullet_2$ from 
Figure~\ref{Fig:Ehrenborg:S3:Poset}.}
\label{Fig:Ehrenborg:S3}
\end{center}
\end{figure}

Ehrenborg and Jung distinguish a block by underlining it.  Thus, the above map is 
\begin{align*}
\Delta_{\vec{c}}&\longrightarrow \Pi_{\vec{c}}^\bullet\\
B_1\Dash B_2\Dash\cdots\Dash B_k&\longmapsto B_1 | B_2|\cdots |\underline{B_k}.
\end{align*}

\begin{figure}[hbt]
\begin{center}
{
\psfrag{c=110}{\tiny{$\vec{c}=(1,1,1,0)$}}
\psfrag{c=120}{\tiny{$\vec{c}=(1,2,0)$}}
\psfrag{c=111}{\tiny{$\vec{c}=(1,1,1)$}}
\psfrag{c=210}{\tiny{$\vec{c}=(2,1,0)$}}
\psfrag{c=30}{\tiny{$\vec{c}=(3,0)$}}
\psfrag{c=21}{\tiny{$\vec{c}=(2,1)$}}
\psfrag{c=12}{\tiny{$\vec{c}=(1,2)$}}
\psfrag{c=3}{\tiny{$\vec{c}=(3)$}}
\includegraphics{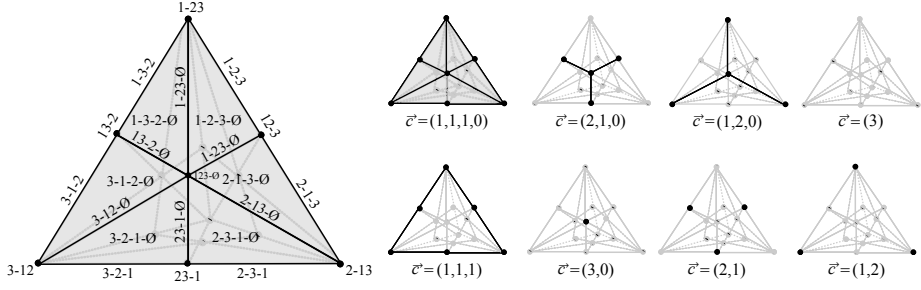}
}
\caption{Ehrenborg and Jung's $\Delta_{\vec{c}}$ for all possible choices of $\vec{c}\vdash 3$ 
(with last part is allowed to be $0$).}
\label{Fig:Ehrenborg}
\end{center}
\end{figure}

The following proposition summarizes the correspondence outlined above.

\begin{proposition}\label{Prop:Conversion}
Let $(W,R,\Lambda)$ be the standard 
system for $\mathfrak S_{n+1}$, and let $\vec{c}=(c_1,c_2,\ldots,c_k)$ be 
a composition of $n$ with last part $c_k$ possibly $0$.  
Then the following diagram composed of $\mathfrak S_n=W_{R\diff \{r_n\}}$-equivariant maps is commutative:
\[\xymatrixcolsep{.7pc}\xymatrixrowsep{3pc}
\xymatrix{
 & \Delta_{ {\rm{Des}}(\vec{c})}^{\{r_n\}}
\ar@{->>}[d]^-{}_-{\Supp} 
\ar@{<->}[rrrrr]^-{\sim}_-{} & & & & &
\Delta_{\vec{c}}\ar@{->>}[d]
& {\scriptstyle{B_1\Dash B_2\Dash \cdots\Dash B_k}} \ar@{|->}[d]^{\Supp}
\\
 & \Pi_{{\rm{Des}}(\vec{c})}^{\{r_n\}} \ar@{<.>}[rrrrr]^-\sim_-{} & & & & &{\Pi_{\vec{c}}^\bullet} & {\scriptstyle{B_1|B_2|\cdots|\underline{B_k}}}
}
\]
\end{proposition}

The main result of~\cite{Ehrenborg} is the following

\begin{theorem}[Ehrenborg and Jung]\label{Thm:EJ}  
Let $\vec{c}=(c_1,c_2,\ldots, c_k)$ be a 
composition of $n+1$ with last part $c_k$ possibly $0$.  Then we have 
the following isomorphism of top (reduced) homology groups as $\mathfrak S_n$-modules:
\[\tilde{H}_{\rm{top}}(\Delta(\Pi^\bullet_{\vec{c}}\diff \{\hat{1}\}))\homotopic_{\mathfrak S_n}\tilde{H}_{\rm{top}}(\Delta_{\vec{c}}).\]
\end{theorem}

\begin{remark}
Proposition~\ref{Prop:Conversion} shows that Theorem~\ref{Thm:EJ} is implied by 
combining Theorem~\ref{C:Thm}\eqref{C:locallyconical} or Theorem~\ref{S:Thm}\eqref{S:locallyconical} below with Theorem~\ref{Equivalence} (or Theorem~\ref{Theorem:Main}).
\end{remark}

\section{Coxeter groups}\label{Section:Coxeter}
The main theorem of this section is that we can apply all previous results 
to finite irreducible Coxeter groups.  

\begin{theorem}\label{C:Thm}
Let $(W,R)$ be a simple system for a 
finite irreducible real reflection group $W$, and let 
$\Lambda=\{\lambda_1,\lambda_2,\ldots,\lambda_\ell\}$ 
consist of one nonzero point from each extreme ray of a Weyl chamber 
corresponding to $R$.  Then the following hold:
\begin{enumerate}[(i)]
\item  $(W,R,\Lambda)$ is strongly stratified.\label{C:stronglystratified}
\item  $\rho(\Delta)$ is homeomorphic to the Coxeter complex of $(W,R)$ via 
radial projection.\label{C:complex}
\item  $\Delta:=\Delta(W,R)$ is homotopy Cohen-Macaulay. \label{C:C-M}
\item  $(W,R,\Lambda)$ is locally conical.\label{C:locallyconical}
\end{enumerate}
\end{theorem}

Properties~\eqref{C:stronglystratified}-\eqref{C:C-M} are well-known.  
The aim of this section 
is to establish~\eqref{C:locallyconical}.

We will need the following

\begin{lemma}\label{Lemma:Proj}
Let $W$ be a finite irreducible real reflection group, and let  
$\lambda_{1},\lambda_{2},\ldots, \lambda_{k}$ be nonzero vectors 
on extreme rays of a fixed Weyl chamber.  Then the orthogonal projection 
of $\lambda_{1}$ onto $\LinSpan( \lambda_2,\ldots, \lambda_k)$ is nonzero.
\end{lemma}

\begin{proof}
This follows from the claim that, for all $i,j$, one has 
$\langle \lambda_i,\lambda_j\rangle>0$ 
for any set $\{\lambda_1,\ldots, \lambda_\ell\}$ of nonzero vectors 
on the extreme rays of a Weyl chamber $C$.  

To see this claim, let $\alpha_1,\ldots,\alpha_\ell$ be the simple system of roots associated 
with $C$.  In particular, $\alpha_i$ is orthogonal to 
\[H_i=\Span(\lambda_1,\ldots,\hat{\lambda}_i,\ldots,\lambda_\ell),\]
and its direction is chosen so that $\langle \lambda_i,\alpha_i\rangle\geq 0$.  
Recall that the $\alpha_i$ form an obtuse basis for $V$, i.e., a basis with
\[\langle \alpha_i,\alpha_j\rangle\leq 0\quad\text{for all $i,j$}.\]
It follows from\cite[Ch.V, \S 3, no. 5, Lemma 6]{Bourbaki} that $C\subseteq \left\{\Sigma_{i}c_i\alpha_i\ :\ c_i\geq 0\right\}$.
This implies the weak inequality, i.e., $\langle \lambda_i,\lambda_j\rangle\geq 0$ for all $i,j$. 
 
One can now obtain the desired strict inequality by using the connectivity of the 
Coxeter diagram for $W$ and the fact that the $\alpha_i$ are obtuse; 
see~\cite[p. 72, no. 8]{Humphreys} for an outline.
\end{proof}

We will also need some basic facts regarding convexity of the  
Coxeter complex; see~\cite{Brown}, particularly Section 3.6, 
for details and a more general treatment.  
Let $W\leq \GL(\mathbb R^\ell)$ be a finite 
real reflection group, and let $\Sigma$ denote its Coxeter complex.  
Recall that $\Sigma$ is a \emph{chamber complex}, meaning that all 
maximal simplices (called \emph{chambers}) are of the same dimension, and any two 
chambers can be connected by a gallery.  Here, a \emph{gallery} connecting two 
chambers $C$, $D$ is a sequence of chambers
\[C=C_0,C_1,\ldots,C_k=D\]
with the additional 
property that consecutive chambers are adjacent, meaning that they 
share a codimension-1 face.

A \emph{root}\footnote{ 
Root vectors of $W$ in $\mathbb R^\ell$ (vectors perpendicular to 
reflecting hyperplanes) are in canonical 
1-1 correspondence with roots of $\Sigma$.  The notion of root extends 
to arbitrary thin chamber complexes by introducing the notion of a folding.  
The terminology is due to Tits, who characterized 
abstract Coxeter complexes as precisely those 
thin chamber complexes with ``enough'' foldings; see~\cite[Sec. 3.4]{Brown}.
}
of $\Sigma$ is the intersection of $\Sigma$ with a 
closed half-space determined by a reflecting hyperplane, 
and a subcomplex of $\Sigma$ is called \emph{convex} if 
it is an intersection of roots.  Each convex subcomplex $\Sigma'$ 
is itself a chamber complex in which any two maximal simplices can be 
connected by a $\Sigma'$-gallery.  Moreover, a chamber subcomplex $\Sigma'$ is convex 
if and only if any shortest $\Sigma$-gallery connecting two chambers of $\Sigma'$ 
is contained in $\Sigma'$.

The main tool for proving 
Theorem~\ref{C:Thm}\eqref{C:locallyconical} is an iterative method for detecting cone 
points.  The following discussion and lemma make this precise.

Choose a nontrivial subset $U\subseteq R$ and consider a nontrivial simplex $\sigma\in\Delta^U$.  
Choose $B$ to be a chamber 
of $\Delta^U$ containing $\sigma*\Lambda_U$.  
From a sequence $b_0,b_1,\ldots, b_{k}$ of distinct vertices of $B$, 
we construct a descending sequence of convex subcomplexes
\[\Delta^U=\Delta_0\supset \Delta_1\supset\cdots\supset \Delta_k \supset \Delta_{k+1},\]
where we set $H_i^B=\Supp(B\diff\{b_i\})$ and define
\[\Delta_i=\left(\bigcap_{j=0}^{i-1} H_j^B\right)\cap \Delta^U.\]
We call the sequence $b_0,b_1,\ldots,b_k$ \emph{cone-approximating} for the triple $(\Delta^U,\sigma,B)$ 
if the following two conditions hold:

\begin{enumerate}
\item  $b_i$ is a cone point of $\Delta_i$ for $0\leq i\leq k$.
\item  $\sigma\in \Delta_{i}$ for $0\leq i\leq k$.
\end{enumerate}

Note that $\Delta_0=\Delta^U$ implies the existence of cone-approximating sequences, 
since $U$ is nontrivial.  
The main result is that any cone-approximating 
sequence can be extended to contain a vertex of $\sigma$:

\begin{lemma}\label{C:approx}
In the above setting, a maximal cone-approximating sequence $b_0,b_1,\ldots,b_m$ for 
$(\Delta^U,\sigma,B)$ has $b_m\in\sigma$. 
\end{lemma}

\begin{proof}
Supposing $b_m\not\in \sigma$, we have that $\sigma\in\lk_{\Delta_m}(b_m)$.  
From this it follows that 
\[\sigma\in  H_m^B\cap \Delta_m=\Delta_{m+1}.\]
First note that $\Delta_{m+1}$ is a convex subcomplex.  Set
\[B_{m+1}=B\diff\{b_0,\ldots, b_m\},\]
the \emph{distinguished chamber} of $\Delta_{m+1}$ containing 
$\sigma$, and let $\tilde{B}_{m+1}$ be another 
chamber of $\Delta_{m+1}$.

Since $\Delta_{m+1}$ is convex, we have that
\[\Delta_{m+1}*\{b_0,\ldots, b_m\}\]
is a convex chamber subcomplex of $\Delta^U$.  Thus, there is a gallery in 
$\Delta_{m+1}*\{b_0,\ldots, b_m\}$ that connects chambers $B_{m+1}*\{b_0,\ldots,b_m\}$ and 
$\tilde{B}_{m+1}*\{b_0,\ldots, b_m\}$.  
Therefore, there is a 
sequence of reflections $\overline{\mathbf r}$ that induces a 
gallery from $B_{m+1}$ to $\tilde{B}_{m+1}$ in $\Delta_{m+1}$.  

Since $\overline{\mathbf r}$ (sequentially) 
stabilizes $\Delta_{m+1}$, meaning that $r_i\Delta_{m+1}=\Delta_{m+1}$ for every reflection 
$r_i$ in $\overline{\mathbf{r}}$, the sequence $\overline{\mathbf r}$ also 
stabilizes both $\Supp(\Delta_{m+1})$ and 
its orthogonal complement in $\Supp(\Delta_m)$.  
Because $b_m$ is fixed by $\overline{\mathbf r}$ and lies strictly on one side of 
$\Supp(\Delta_{m+1})$ in $\Supp(\Delta_m)$, 
it follows that $\overline{\mathbf r}$ 
fixes the orthogonal complement pointwise.  This implies that the orthogonal projection 
${\mathrm{Proj}}^{\Delta_m}_{\Delta_{m+1}}(b_m)$ 
of the vector $b_m$ onto $\Supp(\Delta_{m+1})$ in $\Supp(\Delta_m)$ is fixed by $\overline{\mathbf r}$.  

From Lemma~\ref{Lemma:Proj} we have 
\[\dim\ {\mathrm{Proj}}^{\Delta_{m}}_{\Delta_{m+1}}(b_m)=1.\]
We claim that this projection has nontrivial intersection with the 
cone over $B_{m+1}$.  That is,
\begin{equation}\label{Proj}
\dim\ {\mathrm{Proj}}^{\Delta_{m}}_{\Delta_{m+1}}(b_m)\cap \mathbb R_{>0}B_{m+1}=1.
\end{equation}
To see this, note first that $B_{m+1}*\{b_m\}$ forms a chamber for a Coxeter complex.  
The claim now follows from the fact that 
pairs of walls in a chamber do not intersect obtusely; indeed, writing 
$s_i,s_j$ for the reflections in two distinct walls of a chamber, we have 
$s_i\neq s_j$ and the dihedral 
angle formed by the walls is $\pi/m_{ij}$, where $m_{ij}$ is the order of $s_is_j$.

By~\eqref{Proj}, the line $L=\Span\ {\mathrm{Proj}}^{\Delta_{m}}_{\Delta_{m+1}}(b_m)$ 
intersects a face $F$ of $B_{m+1}$ that is 
minimal in the sense that no proper face of $F$ meets $L$.  Therefore, if $L$ is fixed by a 
reflection, the face $F$ is fixed pointwise by the reflection.  
Choose $b_{m+1}$ to be some vertex of $F$.  We conclude that any gallery in 
$\Delta_{m+1}$ that contains $B_{m+1}$ has $b_{m+1}$ as 
a cone point.  By connectivity, this implies that $\Delta_{m+1}$ has $b_{m+1}$ as a 
cone point.  As we also have $\sigma\in \Delta_{m+1}$ and $b_{m+1}\neq b_{m},\ldots, b_0$, 
we can append $b_{m+1}$ to obtain a longer approximating sequence. 
\end{proof}

\begin{proof}[Proof of Theorem~\ref{C:Thm}\eqref{C:locallyconical}]
Let $U\subseteq R$ be nonempty and $\sigma\in \Delta^U\diff\{\varnothing\}$.  
Choose $B$ to be a chamber of $\Delta^U$ containing $\sigma*\Lambda_U$.  
Consider the Quillen fiber $\Supp(\sigma)\cap \Delta^U$ for the map 
$\Supp:\Face(\Delta^U)\to\Pi^U\diff\{\hat{1}\}$, 
as in~\eqref{UnrestrictedFiber} of Theorem~\ref{Equivalence}, and let $b_0,b_1,\ldots, b_m$ 
be a maximal cone-approximating sequence for $(\Delta^U,\sigma,B)$.
By Lemma~\ref{C:approx}, $\Delta_m$ has cone point $b_m$ that is also a vertex
of $\sigma$.   We want to show that $b_m$ is also a cone point for the Quillen fiber.

Since $\sigma\in \Delta_m$ and $\Delta_m=X\cap \Delta^U$ 
for some particular $X\in \L_W$, it follows that $\sigma\subset X$, and hence
\[\Supp(\sigma)\cap \Delta^U\subseteq \Delta_m.\]
Let $\tau\in \Supp(\sigma)\cap \Delta^U$.  We want to show that the join 
$\tau*\{b_m\}$ is also in the fiber.  Since $\tau*\{b_m\}\in \Delta_m\subseteq \Delta^U$, 
we need only show that $\Supp(\tau*\{b_m\})\subseteq\Supp(\sigma)$.  
But this is clear, since $\tau\subset\Supp(\sigma)$ and $b_m\in \sigma$.
\end{proof}

\section{Shephard groups}\label{Background}

\emph{Shephard groups} form an important class of complex reflection groups.  
They are the symmetry groups of regular complex polytopes, as defined by 
Shephard~\cite{Shephard}.  Here we will follow Coxeter's treatment~\cite{Coxeter}.

Let $\P$ be a finite arrangement of complex affine subspaces of $V$, with partial 
order given by inclusion.  We call its elements \emph{faces}, denoting an 
$i$-dimensional face by $F_i$.  A $0$-dimensional face is called a \emph{vertex}.  
Allowing \emph{trivial faces} $F_n=V$ and $F_{-1}=\varnothing$, all other 
faces are called \emph{proper faces}.  A totally ordered set of proper faces is 
called a \emph{flag}.  
The simplicial 
complex of all flags is called the \emph{flag complex} and is denoted $K(\P)$.  
This is the order complex of $\P$ with its improper faces $\varnothing$ and $V$ omitted.

Such an arrangement $\P$ is a \emph{polytope} if the following hold:

\begin{enumerate}[(i)]
\item  $\varnothing,V\in\P$.\label{poly:1}
\item  If $F_i\subset F_j$ and $|i-j|\geq 3$, then the open interval
\[(F_i,F_j)=\{F\ :\ F_i\subset F\subset F_j\}\] 
is connected, i.e., its Hasse diagram is a connected graph.\label{poly:2}
\item\label{poly:3}  If $F_i\subset F_j$ and $|i-j|\geq 2$, then the open interval $(F_i,F_j)$ 
contains at least two distinct $k$-dimensional faces $F_k,F_k'$ for each $k$ with $i<k<j$.
\end{enumerate}

For $\P$ a polytope, 
note that properties~\eqref{poly:1} and~\eqref{poly:3} enable one to extend any 
partial flag $F_{i_1}\subset F_{i_2}\subset\cdots\subset F_{i_k}$ in $K(\P)$ to a maximal 
flag (under inclusion) of the form
\[F_0\subset F_1\subset\cdots\subset F_{\ell-1}.\]
We call maximal flags in $K(\P)$ \emph{chambers}.  
If the group $W\subset \GL(V)$ of automorphisms of $\P$ acts transitively 
on the chambers of $\P$, then we say that $\P$ is \emph{regular} and that 
$W$ is a \emph{Shephard group}.

The complexifications of the two (affine) arrangements shown in Figure~\ref{starry} are 
examples of regular (complex) polygons.  Both polygons have symmetry group $I_2(5)$, the dihedral group of order 10.
\begin{figure}[hbt]
\begin{center}
\includegraphics{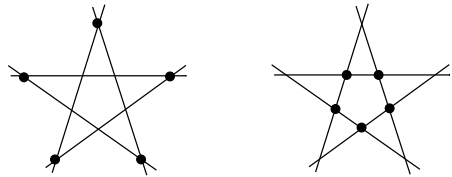}\caption{Two regular polygons with symmetry group $I_2(5)$.}
\label{starry}
\end{center}
\end{figure}

If $\P$ contains a pair of distinct vertices that are at the minimum distance apart 
(among all pairs of distinct vertices) with 
no edge of $\P$ connecting them, then $\P$ is~\emph{starry}; see~\cite[p. 87]{Shephard}.
For example, the first polytope 
of Figure~\ref{starry} is \emph{starry}, whereas the second is \emph{nonstarry}.
From the work of Coxeter, each Shephard group $W$ 
is the symmetry group of two (possibly isomorphic) nonstarry regular complex polytopes; 
see Tables IV and V in~\cite{Coxeter}.
Henceforth, we assume that all polytopes are nonstarry.  Though currently unmotivated, the 
importance of this assumption will be made clear by Theorem~\ref{S:Thm}.

Given a regular complex polytope $\P$ and a choice of maximal flag 
\[\calB=B_0\subset B_1\subset\cdots\subset B_{\ell-1},\]
called the \emph{base chamber}, for each $i$ the group
\[\mathrm{Stab}_W(B_0\subset\cdots\subset \hat{B}_i\subset\cdots\subset B_{\ell-1})\]
is generated by some reflection $r_i$.  Choosing such an $r_i$ for each $i$ 
yields an associated set $R=\{r_0,r_1,\ldots, r_{\ell-1}\}$ that generates $W$.  
We call $R$ a set of \emph{distinguished generators}.  

In the case that 
$\P$ has a real form, each $r_i$ is uniquely determined, and they give the 
usual Coxeter presentation for $W$.  In general, 
Coxeter shows that one can always choose the reflections $r_i$ so that 
for some integers $p_0,p_1,\ldots,p_{\ell-1}, q_0,q_1,\ldots,q_{\ell-2}$ 
the group has the following Coxeter-like presentation with defining relations
\begin{align*}
r_i^{p_i}&=1\\
r_ir_j&=r_jr_i\quad\text{if}\ |i-j|\geq 2\\
\underbrace{r_ir_{i+1}r_i\cdots}_{q_i}&=\underbrace{r_{i+1}r_ir_{i+1}\cdots}_{q_i}.
\end{align*}
These relations are encoded by \emph{symbol} 
\[p_0[q_0]p_1[q_1]p_2\cdots p_{\ell-2}[q_{\ell-2}]p_{\ell-1},\]
which is uniquely determined by $W$ up to reversal.  The corresponding 
nonstarry polytopes are denoted 
\[p_0\{q_0\}p_1\{q_1\}p_2\cdots p_{\ell-2}\{q_{\ell-2}\}p_{\ell-1}
\quad\text{and}\quad p_{\ell-1}\{q_{\ell-2}\}p_{\ell-2}\cdots p_{2}\{q_{1}\}p_{1}\{q_0\}p_0,\] 
the second called the \emph{dual} of the first.  If we denote one of the two polytopes by $\P$, 
then the other is denoted by $\P^*$.  The two polytopes have dual face posets and 
isomorphic flag complexes.

The complete classification 
of Shephard groups is quite short:
\begin{itemize}
\item  The symmetry groups of real regular polytopes, i.e., the Coxeter groups 
with connected unbranched diagrams: types $A_n,B_n=C_n,F_4,H_3,H_4,I_2(n)$.
\item  $p_0[q]p_1$ with $p_0,p_1\geq 2$ not both $2$, and $q\geq 3$ satisfying
\[\frac{1}{p_0}+\frac{1}{p_1}+\frac{2}{q}>1,\]
where $p_0=p_1$ if $q$ is odd.  Using Shephard and Todd's numbering, these groups are
$G_4, G_5,G_6,G_8,G_9,G_{10},G_{14},G_{16},G_{17},G_{18},G_{20},G_{21}$.
\item  $G(r,1,n)=Z_r\wr \mathfrak S_n$ with $r>2$.  The group can be represented 
as $n\times n$ permutation matrices with entries the $r$th roots of unity.
\item  $2[4]3[3]3=G_{26}$
\item  $3[3]3[3]3=G_{25}$
\item  $3[3]3[3]3[3]3=G_{32}$.
\end{itemize}

The following list summarizes notation and assumptions that will remain fixed 
when dealing with Shephard groups.
\begin{itemize}
\item  $W$ is a Shephard group.  
\item  $\P$ is a non-starry regular complex polytope with symmetry group $W$.
\item  $K(\P)$ is the flag complex of $\P$, consisting of all flags of (proper) faces.
\item  $\mathcal B=B_0\subset B_1\subset\cdots\subset B_{\ell-1}$ is a chosen base flag in $K(\P)$.
\item  $R$ is a set of distinguished generators for $W$ corresponding to $\mathcal B$.  
One can choose $R$ to satisfy the presentation found in the classification above, but 
doing so is unnecessary.
\item  Let $U\subseteq R$ with $U=\{r_{i_1},\ldots,r_{i_k}\}$ and $i_1<\cdots< i_k$.  Then 
\[\mathcal B_U\Def B_{i_1}\subset\cdots\subset B_{i_k}.\]
\item  $\L_W$ is the lattice of hyperplane intersections for $W$ under reverse inclusion.
\end{itemize}

\section{Shephard systems}\label{Section:Shephard:2}
The aim of this section is to present an analogue of Theorem~\ref{C:Thm} for Shephard groups.  

If $W$ is the symmetry group of a \emph{real} regular polytope $\P$, then 
the Coxeter complex $\Sigma$ is obtained by intersecting the reflecting hyperplanes 
with the real sphere $\mathbb S^{\ell-1}$.  A radial projection sends $\Sigma$ 
homeomorphically onto the barycentric subdivision of $\P$.
Moreover, it is a geometric realization of the 
(a priori) poset of \emph{standard cosets}
\[\Delta_W=\{gW_{J}\}_{J\subseteq R},\]
ordered by reverse inclusion.  This section presents the analogous picture for Shephard groups, 
as established by Orlik~\cite{OrlikMilnor} and Orlik, Reiner, Shepler~\cite{OrlikReinerShepler}.

The \emph{vertices} of a face $F$ of $\P$ are the vertices of $\P$ lying on $F$, and 
the \emph{centroid} $O_F$ of $F$ is the average of its vertices.   
Centroids play an important role in what follows.

\begin{definition}
A \emph{Shephard system} is a triple $(W,R,\Lambda)$ with the following properties:
\begin{enumerate}[(i)]
\item  $W$ is the symmetry group of a nonstarry regular complex polytope $\mathscr P$.
\item  $R$ is a set of distinguished generators corresponding to a chosen base flag
\[\mathcal B=B_0\subset B_1\subset\cdots\subset B_{\ell-1}.\]
\item  $\Lambda=\{\lambda_0,\lambda_1,\ldots,\lambda_{\ell-1}\}$ is defined by setting 
$\lambda_i=O_{B_i}$.
\end{enumerate}
\end{definition}

We start by observing that such triples are framed systems.

\begin{proposition}  
A Shephard system is a framed system.
\end{proposition}

\begin{proof}
Recall from Section~\ref{Background} that for $0\leq k\leq \ell-1$, the reflection 
$r_k\in R$  
stabilizes 
\[B_0\subset\cdots\subset \hat{B}_k\subset\cdots \subset B_{\ell-1}.\]
In particular, centroid $O_i$ of $B_i$ is fixed by all $r_k$ with $k\neq i$.  
In other words,
\[O_i\in H_0\cap\cdots\cap \hat{H}_i\cap\cdots\cap H_{\ell-1}.\]
\end{proof}

Let $B(\P)$ denote the (topological) subspace of $V$ that consists 
of all real convex hulls of centroids of flags under inclusion, i.e., 
\[B(\P)=\bigcup_{({F_{i_1}\subset\cdots\subset F_{i_k}})\in K(\P)} 
\Hull(O_{F_{i_1}},\ldots, O_{F_{i_k}}).\]
We can now state the analogue of Theorem~\ref{C:Thm} for Shephard groups:

\begin{theorem}\label{S:Thm}
Let $(W,R,\Lambda)$ be a Shephard system.  Then the following hold:
\begin{enumerate}[(i)]
\item  $(W,R,\Lambda)$ is strongly stratified.\label{S:stronglystratified}
\item  $\rho(\Delta)=B(\P)$.\label{S:B}
\item  $\Delta:=\Delta(W,R)$ is homotopy Cohen-Macaulay. \label{S:C-M}
\item  $(W,R,\Lambda)$ is locally conical.\label{S:locallyconical}
\end{enumerate}
\end{theorem}

Note that Figure~\ref{starry} illustrates why the non-starry assumption is necessary; 
indeed, $\rho$ fails to be an embedding in the starry case 
when $\Lambda=\{O_{B_0},O_{B_1}\}$ and $B_0\subset B_1$ is a maximal flag.

The remainder of this section explains~\eqref{S:stronglystratified}-\eqref{S:C-M}, 
while~\eqref{S:locallyconical} will be established in the next section.  During 
our discussion, the reader should take note of our use of Theorems~\ref{barycentric} 
and~\ref{Theorem:OrlikSolomon}, two \emph{uniformly stated} theorems that are \emph{proven 
case-by-case}.  
In particular, Theorem~\ref{barycentric} relies on 
a theorem of Orlik and Solomon that says a Shephard group and an associated 
Coxeter group have the same discriminant matrices, a result relying on the classification 
of Shephard groups; see~\cite{OrlikSolomon:Disc}.  
However, up to the use of Theorems~\ref{barycentric} and~\ref{Theorem:OrlikSolomon}, our 
approach for Theorem~\ref{S:Thm} is case-free.

For each Shephard group, the invariant $f_1$ of smallest 
degree $d_1$ is unique, up to constant scaling.  For example, if 
$W$ has a real form, then $f_1=x_1^2+\cdots+x_\ell^2$ for some suitable set of coordinates.  
The \emph{Milnor fiber of $W$} is defined to be 
$f_1^{-1}(1)$, where $f_1$ is regarded as a map $f_1:V\to \mathbb C$.  
In~\cite{OrlikMilnor}, Orlik constructs a $W$-equivariant strong deformation retraction 
of the Milnor fiber 
$f_1^{-1}(1)$ onto a simplicial complex $\Gamma$ homeomorphic to $B(\P)$, which he shows is 
a geometric realization of the flag complex $K(\P)$.  

\begin{theorem}[Orlik~\cite{OrlikMilnor}]\label{barycentric}  
Let $W\subset {\rm{GL}}(V)$ be a Shephard group with invariant $f_1:V\to \mathbb C$ of smallest 
degree.  Then there exists a simplicial complex $\Gamma\subset f_1^{-1}(1)$ called the 
\emph{Milnor fiber complex} 
containing the vertices of $\P$ such that
\begin{enumerate}
\item
\begin{enumerate}
\item There is an equivariant strong deformation retract $\pi: f_1^{-1}(1)\to\Gamma$.
\item  For each $X\in \L_W$, the set 
$\Gamma_X:=\Gamma\cap X$ is a subcomplex of $\Gamma$, and $\pi$ 
restricts to a strong deformation retract of $f_1^{-1}(1)\cap X$ onto $\Gamma_X$.
\end{enumerate}
\item\label{skeleton}  Let $\Gamma^k$ and $\Gamma_X^k$ denote the $k$-skeleton of each complex.  
Then 
\begin{enumerate}
\item  $\Gamma_X^k=\Gamma^k\cap X$ for all $k$, and
\item  $\Gamma^k\diff\Gamma^{k-1}=\bigcup_{\dim X=k+1}(\Gamma_X^k\diff\Gamma_X^{k-1})$ is a 
disjoint union.
\end{enumerate}
\item\label{realization}
\begin{enumerate}
\item  $\Gamma$ is $W$-equivariantly homeomorphic to $B(\P)$, and
\item  $B(\P)$ is a geometric realization of the flag complex $K(\P)$ via
  \[F_{i_1}\subset\cdots \subset F_{i_k}\longmapsto 
\Hull( O_{F_{i_1}},\ldots, O_{F_{i_k}}).\]
\end{enumerate}
\end{enumerate}
\end{theorem}

Parts (1) and (2) are \cite[Thm 4.1(i)-(ii)]{OrlikMilnor}, while (3) uses the 
proof of \cite[Thm 5.1]{OrlikMilnor}.   A nice discussion of the related theory is 
found in~\cite{OrlikReinerShepler}.

Using the additional property that $f_1^{-1}(1)$ has an isolated critical point at the origin, 
Orlik was able to describe the topology of the flag complex $K(\P)$:  

\begin{theorem}[Orlik~\cite{OrlikMilnor}]\label{CM}  
$K(\P)$ is homotopy Cohen-Macaulay, and is homotopy equivalent 
to a wedge of $(d_1-1)^\ell$-spheres of dimension $\ell-1$.
\end{theorem}

We will establish \eqref{S:stronglystratified} and \eqref{S:C-M} of 
Theorem~\ref{S:Thm} 
by combining Theorem~\ref{barycentric}\eqref{realization} and Theorem~\ref{CM} with 
the following

\begin{theorem}[Orlik, Reiner, Shepler~\cite{OrlikReinerShepler}]\label{Parabolic}  
Let $W$ be a Shephard group of $\P$, and let $\calB$ be a base flag with corresponding 
distinguished generating set $R$.  Then the map
\begin{align*}
\phi:K(\P)&\longrightarrow  \{gW_J\ :\ g\in W,\ J\subseteq R\}\\
g(B_{i_1}\subset\cdots\subset B_{i_s})&\stackrel{\phi}{\longmapsto} 
gW_{R\diff\{r_{i_1},\ldots, r_{i_s}\}}
\end{align*}
is a $W$-equivariant 
isomorphism 
\end{theorem}

The crux of the proof is that 
\[
\mathrm{Stab}_W(B_{j_1}\subset B_{j_2}\subset\cdots\subset B_{j_s})
=
W_{R\diff\{r_{j_1},r_{j_2},\ldots,r_{j_s}\}}
.\]

The type function on $K(\P)$ is naturally given by 
\[\type(B_{i_1}\subset\cdots\subset B_{i_s})=\{r_{i_1},\ldots, r_{i_s}\}.\]

\begin{proof}[Proof of Theorem~\ref{S:Thm}~\eqref{S:B},~\eqref{S:C-M}, and the well-framed 
component of~\eqref{S:stronglystratified}]
Notice that $\rho$ factors as
\[\Delta\stackrel{\phi^{-1}}{\longrightarrow} K(\P)\stackrel{\sim}\longrightarrow 
B(\P)\hookrightarrow V,\]
where $\phi$ is as in Theorem~\ref{Parabolic}, and 
$K(\P)\stackrel{\sim}{\longrightarrow} B(\P)$ is 
provided by Theorem~\ref{barycentric}\eqref{realization}(a).  Hence, the triple 
$(W,R,\Lambda)$ is well-framed and $\rho(\Delta)=B(\P)$.  Employing 
Theorem~\ref{CM} yields~\eqref{S:C-M}.
\end{proof}

All that remains for 
establishing~\eqref{S:stronglystratified}-\eqref{S:C-M} is to show that  
each $X\in \L_W$ contains the image under $\rho$ of a $(\dim X-1)$-simplex of $\Delta$.  
This follows from the following 
beautiful theorem that merges work of Orlik and Solomon~\cite[Thm. 6]{OrlikSolomon} 
and Orlik~\cite[Thm. 4.1(iii)]{OrlikMilnor}; this amalgamation appears in the latter 
paper of Orlik.

\begin{theorem}[Orlik-Solomon]\label{Theorem:OrlikSolomon}  
Let $W$ the symmetry group of a nonstarry regular complex polytope $\P$.  Let 
$X\in \L_W$, and write $\dim X=n$.  Then there exists strictly 
positive integers $b_1^X,\ldots, b_n^X$ such that 
\begin{equation}\label{Equation:OrlikSolomon}
|\Gamma_X^{n-1}\diff\Gamma_X^{n-2}|=(m_1+b_1^X)\cdots(m_1+b_n^X),
\end{equation}
where $m_1=d_1-1$ and $\Gamma_X^{k}$ is the $k$-skeleton of the 
restricted Milnor fiber complex $X\cap \Gamma$.
\end{theorem}

\begin{proof}[Proof of Theorem~\ref{S:Thm}\eqref{S:stronglystratified}]  
Consider $X\in \L_W$ with $\dim X=n\geq 1$.   By Theorem~\ref{Theorem:OrlikSolomon},
$\Gamma_X^{n-1}\diff\Gamma_X^{n-2}$ is nonempty if $m_1\geq 0$, as this implies 
that the right side of~\eqref{Equation:OrlikSolomon} 
is strictly positive.   But this is clear, since 
$d_1=\deg f_1\geq 1$ for any set of basic invariants $f_1,\ldots, f_\ell$.  
(In fact, $d_1\geq 2$, with equality if and only if $W$ is a real reflection group.)

By Theorem~\ref{barycentric}, each 
$(n-1)$-simplex of $\Gamma\cap X$ corresponds to an $(n-1)$-simplex of $B(\P)\cap X$.  
Hence, $(B(\P)\cap X)^{n-1}$ is nonempty.  Since $\rho(\Delta)= B(\P)$, it follows 
by considering dimension that $X=\Supp(\sigma)$ for some $\sigma\in\Delta$.
\end{proof}

\section{Shephard systems are locally conical}

This section is dedicated to proving~\eqref{S:locallyconical} of Theorem~\ref{S:Thm}.  Throughout, $(W,R,\Lambda)$ will be a fixed Shephard system, and 
$F_i$ will denote an $i$-dimensional face of $\P$.  
Because we will need to work with faces of $\P$ instead of centroids,
we start with some straightforward results relating the two.  The most important of these 
results says that centroids of a maximal flag (together with the origin $O_{F_\ell}$)
form an \emph{orthoscheme}; see~\cite[p. 116]{Coxeter}.  
More precisely, we have the following
\begin{proposition}[Coxeter]\label{Prop:Orthoscheme}
Let $\P$ be a regular complex polytope, and let
\[\calF=F_0\subset F_1\subset\cdots\subset F_{\ell-1}\]
be a maximal flag of faces.  Then the vectors
\begin{equation}\label{Eq:Ortho}
O_{F_{\ell-1}}-O_{F_{\ell}},\qquad O_{F_{\ell-2}}-O_{F_{\ell-1}},\qquad\ldots,\qquad 
O_{F_0}-O_{F_{1}}
\end{equation}
form an orthogonal basis for $V$.
\end{proposition}

By taking partial sums in~\eqref{Eq:Ortho}, we obtain the following

\begin{lemma}\label{orthoscheme_basis}
Let $F_0\subset F_1\subset\cdots\subset F_{\ell-1}$ be a maximal flag of a regular polytope $\P$.  
Then the centroids $O_{F_0},O_{F_1},\ldots, O_{F_{\ell-1}}$ form a basis for $V=\mathbb C^\ell$.
\end{lemma}

Two particularly important results follow from Lemma~\ref{orthoscheme_basis}.

\begin{corollary}\label{OrthoSpan}
$F_i=\AffSpan( O_{F_0},\ldots,O_{F_i})$ for all $i\geq 0$.
\end{corollary}

\begin{corollary}\label{FlagIntersection}
If $\calF_1,\calF_2$ are two subflags of a fixed flag $\calF$, then 
\begin{align*}
\LinSpan( \{O_{F}\}_{F\in\calF_1})\cap\LinSpan( \{O_F\}_{F\in\calF_2})
&=\LinSpan( \{O_F\}_{F\in\calF_1\cap\calF_2})
\intertext{and}
\AffSpan( \{O_{F}\}_{F\in\calF_1})\cap\AffSpan( \{O_F\}_{F\in\calF_2})
&=\AffSpan( \{O_F\}_{F\in\calF_1\cap\calF_2}).
\end{align*}
\end{corollary}

We are now ready to present the main tool that will be used in the proof 
of Theorem~\ref{S:Thm}\eqref{S:locallyconical}:

\begin{proposition}\label{unique}
Let $X\in \L_W$ of dimension $k>0$, and suppose that  
$\calF=F_{i_1}\subset \cdots\subset F_{i_k}$ and 
$\calF'=F'_{j_1}\subset\cdots\subset F'_{j_k}$ are two $k$-flags with 
\[
X=\LinSpan( O_{F_{i_1}},\ldots, O_{F_{i_k}})=\LinSpan( O_{F'_{j_1}},\ldots, O_{F'_{j_k}}).
\]
Set $i_{k+1}=\ell$.  For $1\leq s\leq k$, if $F$ is a face such that 
$F_{i_s}\subseteq F\subsetneq F_{i_{s+1}}$,
then one of the following holds:
\begin{enumerate}[(a)]
\item  $F_{j_s}'=F_{i_s}$. 
\item  $F'_{j_t}\not\subseteq F$ for all  $t$.
\end{enumerate}
 \end{proposition}

\begin{proof}
Suppose that $F'_{j_t},F_{i_s}\subseteq F\subsetneq F_{i_{s+1}}$ with $s$ and $t$ maximal.  
Using the definition of a polytope, we can extend 
$F_{i_1}\subset\cdots\subset F_{i_s}\subseteq F\subseteq \cdots\subseteq F_{i_k}$ 
to a maximal flag $\calF$.  By doing so, 
Corollaries~\ref{OrthoSpan} and~\ref{FlagIntersection} imply that
\begin{equation}\label{1}
F\cap X=\AffSpan( O_{F_{i_1}},\ldots, O_{F_{i_s}}).
\end{equation}  
Similarly,
\begin{equation}\label{2}
F\cap X=\AffSpan( O_{F'_{j_1}},\ldots, O_{F'_{j_t}}).
\end{equation}
By comparing dimension, we have that $t=s$.

Our first claim is that $F_{i_s},F'_{j_s}$ are minimal faces 
of $F$ containing $F\cap X$.  Certainly the intersection is 
contained in $F_{i_s}$ (and $F'_{j_s}$).  Moreover, 
it contains the centroid $O_{F_{i_s}}$ (resp. $O_{F'_{j_s}}$), which cannot be contained 
in any proper face of $F_{i_s}$ (resp. $F'_{j_s}$) by Corollary~\ref{OrthoSpan}.

Assume without loss of generality that 
$i_s\leq j_s$.  From the equalities of \eqref{1} and \eqref{2}, we have 
$O_{F'_{j_s}}\in F_{i_s}$.  The definition of a polytope implies the existence 
of a face $\tilde{F}_{j_s}\supseteq F_{i_s}$, which must therefore necessarily contain 
both $O_{\tilde{F}_{j_s}}$ and $O_{F'_{j_s}}$.  We claim that $F'_{j_s}=\tilde{F}_{j_s}$.  This follows 
immediately if $O_{F'_{j_s}}=O_{\tilde{F}_{j_s}}$, since faces and centroids determine 
each other; see Theorem~\ref{barycentric}(3)(b), for example.  The other case is 
slightly more work.

Suppose that $O_{F'_{j_s}}\neq O_{\tilde{F}_{j_s}}$.  Extend $\tilde{F}_{j_s}$ to a maximal flag 
$\tilde{\mathcal F}$ and recall that the centroids for $\tilde{\mathcal F}$ form an orthoscheme.  
Combining Corollary~\ref{OrthoSpan} with Proposition~\ref{Prop:Orthoscheme}, 
it follows that the vector of $O_{\tilde{F}_{j_s}}$ is perpendicular to $\tilde{F}_{j_s}$.  As $O_{F'_{j_s}}$ is also in $\tilde{F}_{j_s}$, this implies 
that $O_{F_{j_s}}-O_{\tilde{F}_{j_s}}$ 
is perpendicular to $O_{\tilde{F}_{j_s}}$.  Hence,
\begin{align*}
|O_{F'_{j_s}}|^2&=|O_{F'_{j_s}}-O_{\tilde{F}_{j_s}}+O_{\tilde{F}_{j_s}}|^2\\
&=\langle O_{F'_{j_s}}-O_{\tilde{F}_{j_s}}+O_{\tilde{F}_{j_s}},O_{F'_{j_s}}-O_{\tilde{F}_{j_s}}+O_{\tilde{F}_{j_s}}\rangle\\
&=|O_{F'_{j_s}}-O_{\tilde{F}_{j_s}}|^2+|O_{\tilde{F}_{j_s}}|^2.
\end{align*}
Because $O_{F'_{j_s}}\neq O_{\tilde{F}_{j_s}}$, we therefore have
\begin{equation*}
|O_{F'_{j_s}}|>|O_{\tilde{F}_{j_s}}|.
\end{equation*}
This contradicts regularity, since there is a unitary $g\in W$ with 
$g\tilde{F}_{j_s}=F'_{j_s}$, i.e.,  with $gO_{\tilde{F}_{j_s}}=O_{F'_{j_s}}$.

Having established that $F'_{j_s}=\tilde{F_{j_s}}\supseteq F_{i_s}$, the minimality 
of $F'_{j_s}$ forces equality.
\end{proof}

\begin{proof}[Proof of Theorem~\ref{S:Thm}\eqref{S:locallyconical}]
Let $U\subseteq R$ be nonempty.  Identifying $\Delta$ with $K(\P)$, we have 
$W_{R\diff U}$ corresponds to $\mathcal B_U$, and 
\[\Delta^U=\St_{\Delta}(W_{R\diff U})=\St_{K(\P)}(\mathcal B_U),\]
with $\Supp:\Face(\Delta^U) \longrightarrow \Pi^U\diff\{\hat{1}\}$ given by
$F_{i_1}\subset\cdots \subset F_{i_s}\longmapsto \LinSpan( O_{F_{i_1}},\ldots, O_{F_{i_s}})$.

Let $X\in \pPi\diff\{\hat{1}\}$.  
We claim that for some face $F_t$ of $\P$, the Quillen fiber 
\[\Phi=X\cap \Delta^U\] 
has $F_t$ as a cone point.
This implies the desired claim of $(W,R,\Lambda)$ being locally conical.

From the definition of $\Pi^U$ and 
Theorem~\ref{barycentric}\eqref{realization}, 
we can write 
\[X=\LinSpan( O_{F_{t_1}},\ldots,O_{F_{t_k}})\]
for some $k$-flag $\calF=F_{t_1}\subset\cdots\subset F_{t_k}$ in $\Delta^U$.  
Because $\calF$ can be extended to a flag in $\Delta$ containing $\mathcal B_U$, 
we can fix a nontrivial face $F\in\mathcal B_U$ and choose $m$ such that one of the following 
holds:

\begin{enumerate}[(a)]
\item  $F_{t_m}$ is the maximal face of $\calF$ that is weakly contained in $F$, or
\item  $F_{t_m}$ is the minimal face of $\calF$ that weakly contains $F$.
\end{enumerate}

We claim that $F_{t_m}$ is a cone point of $\Phi$.  To establish this, 
suppose $\mathcal F'$ is another $k$-flag in $\Delta^U$ with support $X$.  Using 
Lemma~\ref{unique}, we will show that 
$F_{t_m}\in \mathcal F'.$

Suppose first that we are in case (a), i.e.,   
$F_{t_m}\subseteq F$.  
Because $\calF'$ can be extended to a maximal flag containing $\mathcal B_U$, either
\begin{enumerate}[(1)]
\item  an element of $\calF'$ is weakly contained in $F$, or
\item  each element of $\calF'$ strictly contains $F$.
\end{enumerate}
However, (2) is not a valid possibility.  Supposing otherwise, we can extend 
$\calF'$ by $F_{t_m}$ to 
obtain a strictly larger flag with support equal to $X$.  Considering dimension yields 
a contradiction.  
In the case of (1), Lemma~\ref{unique} shows that $F_{t_m}\in \calF'$.  

If we are instead in case (b), i.e., $F_{t_m}\supseteq F$, we can use the same argument, applied to 
the dual polytope $\P^*$, to conclude that
$F_{t_m}\in\calF'$. 
\end{proof}

\section{Additional properties of ribbon representations}
The aim of this section is to relate ribbon representations of Shephard groups to 
the group algebra, the exterior powers of the reflection representation, and to the coinvariant algebra.  

Recall that the group algebra $\mathbb C[G]$ of a finite group $G$ over $\mathbb C$ is semisimple, and thus 
the Grothendieck group $R(G)$ of the category of finite dimensional $G$-representations is the free $\mathbb Z$-module with basis the 
isomorphism classes of irreducible $G$-modules.  The map sending 
each irreducible module to its character linearly extends to an isomorphism of $\mathbb C\otimes_{\mathbb Z}R(G)$ and the 
space of complex class functions on $G$.  
In fact, this is an isomorphism in the category of finite-dimensional Hilbert spaces, 
where the form $\langle -,-\rangle$ on $\mathbb C\otimes_{\mathbb Z}R(G)$ is that for which the irreducible $G$-modules form an orthonormal basis, and 
the form $\langle-,-\rangle$ on class functions is given by $\langle \varphi,\psi\rangle=\frac{1}{|G|}\sum_{g\in G}\varphi(g)\overline{\psi(g)}$.  
In what follows, we make no notational distinction between elements of $\mathbb C\otimes_{\mathbb Z}R(G)$ and class functions.

\subsection{Solomon's group algebra decomposition}
Let $(W,R)$ be a simple system, and $\epsilon$ be the sign representation, defined by $\epsilon(r)=-1$ for $r\in R$.  
In~\cite{Solomon:Decomp}, Solomon showed that
\begin{equation}\label{SolomonDecomp}\mathbb C W=\bigoplus_{T\subseteq R} \mathbb C Wc_T\quad\text{and}\quad
\mathbb C W c_T\cong_W\sum_{T\subseteq J\subseteq R}(-1)^{|J\diff T|}\Ind_{W_J}^W \triv
\end{equation}
for $c_T:=a_T b_{R\diff T}$, where
\[a_T=\frac{1}{|W_T|}\sum_{g\in W_T} g\quad\text{and}\quad b_{R\diff T}=\frac{1}{|W_{R\diff T}|}\sum_{g\in W_{R\diff T}} \epsilon(g) g.\]  

The $W$-modules $\mathbb C W c_T$ are known as 
\emph{ribbon representations} due to their alternate description when $W=\mathfrak S_n$ and $r_i=(i,i+1)$ for $1\leq i\leq n-1$.  
Considering this case, let $T\subseteq R$, write 
$R\diff T=\{r_{i_1},\ldots, r_{i_j}\}$ with $i_1<\cdots< i_j$, and let $\lambda/\mu$ be the ribbon skew shape corresponding to composition $(i_1,i_2-i_1,\ldots, i_j-i_{j-1},n-i_j)$; 
see Figure~\ref{Tableau:Fig}.  
Filling $\lambda/\mu$ with $1,\ldots, n$ in increasing order from southwest to northeast produces a tableau 
whose rows are stabilized by $W_T$ and whose columns are stabilized by $W_{R\diff T}$.  Thus, 
$c_T$ is the \emph{Young symmetrizer} $c_{\lambda/\mu}$, and hence $\mathbb C W c_T$ is the 
$\mathfrak S_n$-Specht module of ribbon skew shape $\lambda/\mu$.
\begin{figure}[hbt]
\begin{center}
\includegraphics{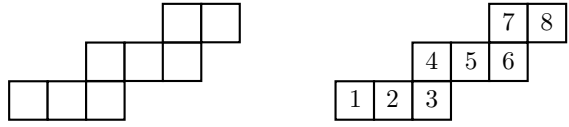}
\caption{The ribbon skew shape and filling for composition $(3,3,2)$.}
\label{Tableau:Fig}
\end{center}
\end{figure}  

For a simple system or Shephard system $(W,R)$, define $\chi^T=\tilde{H}_{|T|-1}(\Delta_T)$ for $T\subseteq R$.  
We call these \emph{ribbon representations}, noting that in the case of Coxeter groups 
\[\mathbb C c_T\cong _W \sum_{T\subseteq J\subseteq R}(-1)^{|J\diff T|}\Ind_{W_J}^W \triv =\sum_{J\subseteq R\diff T} (-1)^{|(R\diff T)\diff J|}\Ind_{W_{R\diff J}}^W \triv \cong_W \chi^{R\diff T}.\]
Regarding Solomon's group algebra decomposition, applying M\"obius inversion to the equality $\chi^R=\sum_{T\subseteq R} (-1)^{|R\diff T|}\Ind_{W_{R\diff T}}^W \triv$ yields

\begin{theorem}\label{CS:Decomp}
For $(W,R)$ a simple system or Shephard system, $\mathbb C W\cong_W\bigoplus_{T\subseteq R} \chi^T$.
\end{theorem}

Another extension of a main theorem in~\cite{Solomon:Decomp} is

\begin{theorem}\label{Theorem:Steinberg}
Let $(W,R)$ be a simple system or Shephard system.  For $T\subseteq R$, the ribbon representation 
$\chi^T$ contains a unique irreducible submodule isomorphic to $\bigwedge^{|T|}V$ and has 
no submodule isomorphic to $\bigwedge^p V$ for $p\neq |T|$.
\end{theorem}

One can follow the same proof as in~\cite{Solomon:Decomp}, replacing~\eqref{SolomonDecomp} with Theorem~\ref{CS:Decomp}.  
However, for the reader's convenience, we supply a simple and significantly shorter argument.

\begin{proof}[Proof of Theorem~\ref{Theorem:Steinberg}]  
Since $ \Wedge^{|T|}V$ occurs in $\CC W$ with multiplicity equal to its
dimension $\binom{|R|}{|T|}$, by Theorem~\ref{CS:Decomp} it will suffice
to show $\langle \chi^T,\, \Wedge^{|T|}V\rangle =1$ for each $T\subseteq R$.  
Fixing $T\subseteq R$, so that $\chi^T=\sum_{J\subseteq T}(-1)^{|T\diff J|} \Ind_{W_{R\diff J}}^W \triv$,
it suffices to show for each subset $J\subseteq T$ that
\begin{equation}
\label{desired-Kronecker-delta}
\begin{aligned}
\left\langle \Ind_{W_{R\diff J}}^W \triv\ ,\ \Wedge^{|T|}V\right\rangle 
=
\begin{cases} 1 & \text{if}\ J=T,\\ 0 & \text{otherwise.}\end{cases}
\end{aligned}
\end{equation}
If $J=R$, then $\Ind_{W_{R\diff J}}^W \triv$ is the regular 
representation and $\dim \Wedge^{|T|}V=1$, so~\eqref{desired-Kronecker-delta} follows.  

Assuming that $J\neq R$, there are disjoint nonempty sets $I_j\subseteq R\diff J$ such that 
\[
W_{R\diff J}\cong W_{I_1}\times\cdots \times W_{I_n}
\] 
with each $W_{I_j}$ acting irreducibly on $V_j:=(\bigcap_{s\in I_j}H_s)^\perp$, thus yielding a decomposition 
\[
V=V^{W_{R\diff J}}\oplus V_1\oplus\cdots\oplus V_n
\] 
on which each $W_{I_j}$ acts trivially on all factors except $V_j$; see~\cite[Ch. 1]{Lehrer}.
The resulting decomposition of the exterior power
\[
{ \Wedge^{|T|}} V \cong \bigoplus_{i_1+\cdots+i_n+m=|T|}
{ \Wedge^{m}}(V^{W_{R\diff J}})\otimes 
{ \Wedge^{i_1}}V_1\otimes\cdots\otimes 
{ \Wedge^{i_n}}V_n,
\]
combined with Frobenius reciprocity, implies that 
\begin{align*}
\left\langle\Ind_{W_{R\diff J}}^W \triv\ ,\ { \Wedge^{|T|}} V\right\rangle 
&=
\sum_{m=0}^{|J|}\binom{|J|}{m}\sum_{i_1+\cdots+i_n=|T|-m}\ \prod_{j=1}^n \left\langle \triv\ ,\ { \Wedge^{i_j}}V_j \right\rangle_{W_{I_j}}.
\end{align*}
By a theorem of Steinberg~\cite[Ch. 5, \S 2, Exercise 3]{Bourbaki}, the $W_{I_j}$-modules   
$\Wedge^{k}V_j$ are irreducible and distinct for $0\leq k\leq\dim V_j$.  Hence 
$\langle \triv\, ,\, { \Wedge^{i_j}}V_j\rangle_{W_{I_j}} = \delta_{i_j,0}$ and~\eqref{desired-Kronecker-delta} follows.
\end{proof}

\subsection{Expressing the ribbon decomposition of the coinvariant algebra}
A central object in invariant theory is the coinvariant algebra $S/S_+^W$ of a 
finite subgroup $W\subset \GL(V)$.  This is the graded quotient of $S$ by the ideal 
$S_+^W$ generated by homogeneous invariants of positive degree.  Recall that $W$ is a reflection group 
if and only if $S/S_+^W$ affords the regular representation as an ungraded module.

This section concerns the decomposition 
of the coinvariant algebra of a Shephard group $W$ into ribbon representations.  More precisely, 
we give a determinantal expression for the multivariate generating function
\[W(\mathbf{t},q){\Def}\sum_{T\subseteq R} \langle \chi^T,S/S_+^W\rangle(q)\ \mathbf{t}^T,\]
where $\mathbf{t}^T$ is defined below, and $\langle \chi,M\rangle(q)$ denotes 
the graded inner product $\sum_{d\geq 0}\langle \chi,M_d\rangle q^d$
of an element $\chi\in\mathbb C\otimes_{\mathbb Z}R(W)$ and a graded $W$-module $M=\bigoplus_{d=0}^\infty M_d$.

For a Shephard system $(W,R)$ and subset $J\subseteq R$, define
\[\mathbf{t}^J=\prod_{r_i\in J} t_i,\quad
(\mathbf{1}-\mathbf{t})^J=\prod_{r_i\in J} (1-t_i),\quad
W_{J}(q)=\Hilb(S/S_+^{W_J},q),\quad
W_{[i,j]}=W_{\{r_k\ :\ i\leq k\leq j\}},
\]
where $W_\varnothing$ is the trivial subgroup, we set $W(q)=W_R(q)$, and the \emph{Hilbert series} of 
a graded module $M=\bigoplus_{d=0}^\infty M_d$ is the formal power series $\Hilb(M,q):=\sum_{d\geq 0} M_d q^d$.

Recall that the generators of a Shephard system $(W,R)$ inherit an indexing from the associated chosen flag of faces 
$B_0\subset\cdots\subset B_{\ell-1}$.  However, in what follows it will be convenient to shift all indices by 1, thus writing  
$R=\{r_1,\ldots, r_\ell\}$.

\begin{theorem}\label{LinearDiagram}
Let $(W,R)$ be a Shephard system.  Then
\begin{equation}\label{Hilb:Series}W(\mathbf{t},q)=
W(q)\cdot\det\left[\begin{matrix}
1     & \frac{1}{W_{[1,1]}(q)} & \frac{1}{W_{[1,2]}(q)} &\cdots  &\frac{1}{W_{[1,\ell]}(q)} \\
t_1-1 & t_1                   & \frac{t_1}{W_{[2,2]}(q)}  &\cdots   &\frac{t_1}{W_{[2,\ell-1]}(q)} \\
0     & t_2-1                 &  t_2                      &\cdots  &\frac{t_2}{W_{[3, \ell-2]}(q)} \\
\vdots & \ddots              & \ddots                    &\ddots &\vdots \\
0      & \cdots              & 0                        &t_{\ell}-1   &t_{\ell}
\end{matrix}\right].\end{equation}
\end{theorem}

\begin{remark}
For a (finite) Coxeter system $(W,R)$, the \emph{multivariate $q$-Eulerian distribution} is defined to be
\[{\rm{Eul}}(\mathbf{t},q)=\sum_{g\in W}\mathbf{t}^{\mathrm{Des}(g)}q^{\ell(g)},\]
where $\ell(g)=\min\{m : r_{i_1}\cdots r_{i_m}=g\}$ is the usual length function on Coxeter groups, and 
$\mathrm{Des}(g)=\{r\in R\ :\ \ell(gr)<\ell(g)\}$ is the descent set of $g$.  
In the case of real Shephard groups, Reiner~\cite{Reiner}, following Stembridge~\cite{Stembridge}, established Theorem~\ref{LinearDiagram} with 
$\rm{Eul}(\mathbf{t},q)$ in place of $W(\mathbf{t},q)$.  
Thus, 
\begin{equation}\label{Eul:W(t,q)}{\rm{Eul}}(\mathbf{t},q)=W(\mathbf{t},q)\end{equation} 
for real Shephard groups.  Extending~\eqref{Eul:W(t,q)} to other Shephard groups is a problem of 
considerable interest.
\end{remark}

\begin{proof}[Proof of Theorem~\ref{LinearDiagram}]
Fix $T\subseteq R$ and consider the coefficient of $\mathbf{t}^T$ in $W(\mathbf{t},q)$.  
By Frobenius reciprocity, we have
\begin{align*}
\langle \chi^T, S/S_+^W\rangle(q) &= \sum_{J\subseteq T}(-1)^{|T\diff J|}\langle \Ind_{W_{R\diff J}}^W \triv,S/S_+^W\rangle(q)\\
&=\sum_{J\subseteq T}(-1)^{|T\diff J|}\langle \triv,S/S_+^W\rangle_{W_{R\diff J}}(q).
\end{align*}

Recall that for a reflection group $G\subset\GL(V)$, one has $S\cong S^G\otimes_{\mathbb C}S/S_+^G$ as graded $G$-modules.  
It follows that $W_J(q)=\frac{\Hilb(S,q)}{\Hilb(S^{W_J},q)}$ for any $J\subseteq R$, and that 
the graded $W$-module $S/S_+^W$ affords graded character $\chi(g)=\sum_{d\geq 0}{\mathrm{Tr}}(g\mid_{(S/S_+^W)_d})q^d=\frac{1}{\det(1-gq)}\frac{1}{\Hilb(S^W,q)}$.  
Therefore, 
\[\langle \triv,S/S_+^W\rangle_{W_{R\diff J}}(q)
=
\frac{1}{\Hilb(S^W,q)}\langle \triv,S\rangle_{W_{R\diff J}}(q)
=
\frac{\Hilb(S^{W_{R\diff J}},q)}{\Hilb(S^W,q)}=\frac{W(q)}{W_{R\diff J}(q)}.\]

Consider now the right-hand side of~\eqref{Hilb:Series}.  From the usual permutation expansion of the determinant, we have the 
following general equation:
\[\det\begin{bmatrix} a_{01} & a_{02} & a_{03} &\cdots & a_{0,n+1} \\ a_{11} & a_{12} &  a_{13} &\cdots & a_{1,n+1} \\
0 & a_{22} & \ddots &\ddots & \vdots \\ \vdots & \ddots & \ddots &\ddots &\vdots  \\
0 & \cdots & 0 & a_{nn} & a_{n,n+1}\end{bmatrix}=\sum_{1\leq i_1<i_2<\cdots< i_{j}\leq n} (-1)^{n-j} a_{0,i_{1}}a_{i_1,i_2}\cdots a_{i_j,n+1}\prod_{\stackrel{1\leq i\leq n}{i\neq i_1,\ldots, i_j}} a_{ii}.\]
Thus, the right-hand side of~\eqref{Hilb:Series} is equal to
\[W(q) \sum_{1\leq i_1<\cdots<i_j\leq \ell} (-1)^{\ell-j}\frac{t_{i_1}\cdots t_{i_j}}{W_{[1,i_1-1]}(q) W_{[i_1+1,i_2-1]}(q)\cdots W_{[i_j+1,\ell]}(q)}
\prod_{\stackrel{1\leq i\leq \ell}{i\neq i_1,\ldots, i_j}}(t_i-1),
\]
which, using the fact that $r_ir_j=r_jr_i$ for $|i-j|\geq 2$, can be written as
\[W(q)\sum_{J\subseteq R}(-1)^{|R\diff J|}\frac{\mathbf{t}^J}{W_{R\diff J}(q)}(\mathbf{t}-\mathbf{1})^{R\diff J}.\]
Taking 
the coefficient of $\mathbf{t}^T$ yields  $\sum_{J\subseteq T} (-1)^{|T\diff J|}\frac{W(q)}{W_{R\diff J}(q)}$.
\end{proof}

\section{Remarks and questions}

\subsection{An interesting family of well-framed systems}

In Section~\ref{Section:Well-framed} we introduced well-framed and strongly stratified systems $(W,R,\Lambda)$.  
The well-framed systems subsequently studied were additionally strongly stratified, locally conical, and 
produced Cohen-Macaulay complexes $\Delta(W,R)$.  However, there are well-framed systems $(W,R)$ 
lacking many of these properties, including that of $\Delta(\Pi^U_T\diff\{\hat{1}\})$ being homotopy equivalent to $\Delta^U_T$.  
We discuss some of these here.

Consider $\P_n$ the boundary of the standard $n$-simplex having vertices 
labeled by  the set $\{1,2,\ldots, n+1\}$.  Its symmetry group is $\mathfrak S_{n+1}$, and a 
minimal generating set of reflections corresponds to a spanning tree $T$ of the complete graph $K_{n+1}$ on $\{1,2,\ldots, n+1\}$ by 
identifying an edge $ij$ of $T$ with the transposition $(i,j)$; see~\cite{BabsonReiner}.  We will focus on the generating 
set
\[R^\star_n:=\{\underbrace{(1,n+1)}_{r_1},\underbrace{(2,n+1)}_{r_2},\ldots, \underbrace{(n,n+1)}_{r_n}\}\]
of $\mathfrak S_{n+1}$ that corresponds to the star graph with center $n+1$.   We leave the proof of the following result to the reader.

\begin{proposition}\label{Prop:Star}
Let $\P_n$ be as above.  Let $v_1,v_2,\ldots, v_{n+1}$ be the vertices of $\P_n$, indexed so that $v_i$ corresponds to the 
vertex labeled by $i$.  Let $\alpha_1,\alpha_2,\ldots,\alpha_n\in\mathbb C$ 
be pairwise linearly independent over $\mathbb R$, and  set $\Lambda:=\{\alpha_1 v_1,\ldots, \alpha_n v_n\}$.
Then $(\mathfrak S_{n+1},R^\star_n,\Lambda)$ is well-framed.
\end{proposition}

For $n\geq 3$, a system $(\mathfrak S_{n+1},R^\star_n,\Lambda)$ as in Proposition~\ref{Prop:Star} 
is neither locally conical nor strongly stratified.  Moreover, for $n\geq 4$, it is no longer true that 
$\Delta^U_T$ is homotopy equivalent to $\Delta(\Pi^U_T\diff\{\hat{1}\})$ for every $U,T\subseteq R$ with $U$ nonempty.  These assertions are straightforward after 
first considering

\begin{example}  
Let $(\mathfrak S_5,R^\star_4,\Lambda)$ be a system as in Proposition~\ref{Prop:Star} with $n=4$.  The link of 
a $1$-simplex in $\Delta$ has 6 vertices supporting 3 lines; see Figure~\ref{Fig:TriangleExample}.  Thus, the system is 
not locally conical.  Moreover, it fails to be strongly stratified, since the vertices of $\Delta$ support 
only the 5 complex lines spanned by the vertices of $\P_4$.  

The link $\Delta^{\{r_4\}}_{\{r_1,r_2,r_3\}}=\lk_{\Delta}(W_{R\diff\{r_4\}})$ of a vertex is isomorphic to $\Delta(\mathfrak S_{4},R^\star_3)$.  The latter 
is isomorphic to the $3\times 4$ \emph{chessboard complex}, which is known to be a 2-torus; 
see~\cite[Example 3.1]{BabsonReiner} and~\cite[p. 30]{Chess}.  On the other hand, $\Pi^{\{r_4\}}_{\{r_1,r_2,r_3\}}\diff\{\hat{1}\}$ is easily seen to be 
the face poset of the barycentric subdivision of $\P_3$, hence its order complex is a 2-sphere.
\end{example}

\subsection{The remaining reflection groups}\label{Section:Well-generated_Exceptionals}
This subsection provides evidence that our main results partially extend to the remaining well-generated 
reflection groups.  

\begin{theorem}\label{Theorem:Well-Generateds}
If $(W,R)$ is a well-generated system, then $\Delta(W,R)$ is a simplicial complex.
\end{theorem}

Our proof relies on each subgroup $W_J$ being \emph{parabolic}, meaning 
the pointwise stabilizer of a subset $X\subseteq V$.  The following can be obtained from the 
classification by considering cases.

\begin{theorem}\label{Theorem:Parabolic}
Let $W\subset \GL(V)$ be well-generated by $R$.  Then each standard parabolic $gW_Jg^{-1}$ is parabolic.  More precisely, 
$gW_Jg^{-1}=\Stab(g\cdot\cap_{r\in J} H_r)$.
\end{theorem}

\begin{proof}[Proof of Theorem~\ref{Theorem:Well-Generateds}]  
To establish the intersection property (see Section~\ref{Section:Homology}), first note that
$\Stab(X_1)\cap \Stab(X_2)=\Stab(\Span(X_1,X_2))$ for $X_1,X_2\subseteq V$.  Thus, 
\begin{eqnarray*}
\bigcap_{r\in R\diff J} W_{R\diff\{r\}}&=&\bigcap_{r\in R\diff J} \Stab(\cap_{s\in R\diff\{r\}} H_s)\\
&=&\Stab(\Span(\{\cap_{s\in R\diff\{r\}} H_s\ |\ r\in R\diff J\}))\\
&=&\Stab(\cap _{r\in J} H_r)\\
&=&W_J,
\end{eqnarray*}
where the third equality follows from the obvious inclusion by comparing dimensions. 
\end{proof}

\begin{remark}
Consider the rank 1 reflection group $Z_6:=G(6,1,1)$ generated by a primitive 6th root of unity $\zeta$.  
Observe that $R=\{\zeta^2,\zeta^3\}$ is a minimal generating set with $|R|>\dim V$ and $\Delta(Z_6,R)$ 
simplicial.  Thus, a well-generated group $W$ and minimal generating set $R$ may 
yield a simplicial complex even when $W$ is not well-generated by $R$.  
\end{remark}

\begin{question}  
Let $W$ be a well-generated group, and let $R$ be a minimal generating set of reflections under inclusion.  Is 
$\Delta(W,R)$ necessarily a simplicial complex?
\end{question}

\begin{question}
For which non-well-generated groups $W$ is $\Delta(W,R)$ a simplicial complex for some $R$?
\end{question}

Let $(W,R)$ be a well-generated system.  Motivated by~\eqref{Pi:Rewrite}, define
\[\Supp:\Delta(W,R)\to \L_W\quad\text{by}\quad gW_J\mapsto V^{gW_J g^{-1}}\]
and let
\[\Delta^U_T\ \Def\ \St_{\Delta(W,R)}(W_{R\diff U})\mid_T\quad\text{and}\quad \Pi^U_T\ \Def\ \{V^{gW_J g^{-1}}\ :\ gW_J\in \Delta^U_T\}\]
for $U,T\subseteq R$.  Recall that we identify $\Delta^U_T$ with the poset of faces of a simplicial complex, and thus 
$\Face(\Delta^U_T)$ is obtained from $\Delta^U_T$ by removing its unique bottom element $W$.  
Call $(W,R)$ \emph{(abstractly) locally conical} if for each $U,T\subseteq R$ with $U$ nonempty,  every 
Quillen fiber of $\Supp:\Face(\Delta^U_T)\to \Pi^U_T\diff\{\hat{1}\}$ has a cone point.  
Note that if $(W,R)$ is (abstractly) locally conical, then $\Delta(\Pi_T^U\diff\{\hat{1}\})$ is $W_{R\diff U}$-homotopy equivalent 
to $\Delta^U_T$ for all $U,T\subseteq R$ with $U$ nonempty.

\begin{conjecture}\label{Conjecture:Conical}
For each well-generated reflection group $W$, there exists a well-generating $R$ for which $(W,R)$ is 
(abstractly) locally conical.
\end{conjecture}

Further, we predict the following partial extension of Theorems~\ref{C:Thm} and~\ref{S:Thm}.

\begin{conjecture}\label{Conjecture:Main}
For each well-generated reflection group $W$, there exists a generating set $R$ and a frame $\Lambda$ such that
\begin{enumerate}[(i)]
\item  $|R|=\dim V$.
\item  $(W,R,\Lambda)$ is strongly stratified.
\item  $(W,R,\Lambda)$ is locally conical.
\end{enumerate}
\end{conjecture}

\subsection{Shellability}

It is well-known~\cite{Bjorner:Shelling} that the Coxeter complex $\Gamma$ for a finite 
Coxeter group
is \emph{shellable}, meaning that  
its facets can be ordered $F_1,F_2,\ldots, F_k$ so that the subcomplex $F_j\cap\left( \cup_{i=1}^{j-1} F_i\right)$
is pure of dimension $\dim \Delta-1$ for all $j\geq 2$.  

The question of whether the flag complex $K(\P)$ of a regular complex polytope $\P$ is lexicographically shellable 
appears in~\cite[Question 16]{OrlikReinerShepler} and~\cite[p. 32]{Bjorner:Nonpure}.  
By Section~\ref{Section:Shephard:2}, $K(\P)$ is isomorphic to $\Delta(W,R)$ for $(W,R)$ a Shephard system for $\P$.  
It is straightforward to shell those of rank 2, as they are 
connected graphs, and it is also straightforward for $G(r,1,n)=\mathbb Z_r\wr\mathfrak S_n$.  
Those of Coxeter type are shellable, as mentioned in Question~\ref{Q:Shelling}.  The 
author used a computer to produce shellings in the remaining cases: 

\begin{theorem}\label{ShellingThm}  Let $(W,R)$ be a Coxeter or Shephard system.  Then $\Delta(W,R)$ 
is shellable.
\end{theorem}

\begin{question}\label{Q:Shelling}  
Is there a uniform way of shelling the flag complex $K(\P)$ of a regular complex polytope?  
This would give a more direct proof that $K(\P)$ is homotopy Cohen-Macaulay.
\end{question}

The following was inspired by a personal communication with Taedong Yun and~\cite[Section 8]{Ehrenborg}.

\begin{question}  Let $(W,R)$ be a Coxeter or Shephard system.  Is $\Pi^U_T\diff\{\hat{1}\}$ shellable 
for all $U,T\subseteq R$?
\end{question}

\section{Acknowledgments}
This work forms part of the author's doctoral thesis at the University of 
Minnesota, supervised by Victor Reiner, whom the author thanks 
for suggesting some of these questions.  He is also grateful to Marcelo Aguiar and Volkmar Welker 
for their helpful suggestions and corrections, Richard Ehrenborg and JiYoon Jung for making their work available, 
Stephen Griffeth, Jia Huang, Gus Lehrer, and Vivien Ripoll for helpful conversations, 
and Michelle Wachs for writing a set of truly wonderful notes~\cite{Wachs:Tools}.
{}
\end{document}